\documentclass[a4paper,10pt]{article}

\def\ifundefined#1{\expandafter\ifx\csname#1\endcsname\relax}

\usepackage{graphicx}
\usepackage{graphics}
\usepackage{mathrsfs}
\usepackage{float}
\usepackage{rotating}

\usepackage[T1]{fontenc}
\usepackage[english]{babel}
\usepackage[utf8]{inputenc}
\usepackage{amsthm,amsmath,amssymb,upref,amscd,amsfonts}  
\usepackage{euscript}
\usepackage{epsfig}
\usepackage{ae}
\usepackage{aecompl}
\usepackage{esint}
\usepackage[metapost]{mfpic}
\usepackage{calc}
\usepackage{enumerate}
\usepackage{enumitem}
\usepackage{url}
\usepackage[affil-it]{authblk}

\theoremstyle{plain}
\newtheorem{theorem}{Theorem}[section]
\newtheorem{proposition}[theorem]{Proposition}
\newtheorem{lemma}[theorem]{Lemma}
\newtheorem{corollary}[theorem]{Corollary}
\newtheorem{theorem*}{Theorem}

\theoremstyle{definition}
\newtheorem{remark}[theorem]{Remark}

\numberwithin{equation}{section}

\mathchardef\sa="303A
\renewcommand{\epsilon}{\varepsilon}

\newcommand{\R}{\ensuremath{\mathbf{R}}}

\newcommand{\ds}{\ensuremath{\displaystyle}}

\newcommand{\ip}[3]{\ensuremath{( {#1}, \, {#2} )}}

\newcommand{\lm}[1][m]{\ensuremath{\lambda_{#1}}}
\newcommand{\um}[1][m]{\ensuremath{\mu_{#1}}}
\newcommand{\tk}[1][k]{\ensuremath{\tau_{#1}}}
\newcommand{\eps}{\ensuremath{\epsilon}}
\newcommand{\Jm}[1][m]{\ensuremath{J_{#1}}}

\newcommand{\tke}[1][{}]{\ensuremath{\widehat{\tau_{#1}}}}

\newcommand{\kk}[1][k]{\ensuremath{\kappa_{#1}}}

\newcommand{\Llm}[1][m]{\ensuremath{\Lambda_{#1}(\Oe)}}
\newcommand{\Lum}[1][m]{\ensuremath{\Lambda_{#1}(\Ot)}}

\newcommand{\vp}{\ensuremath{\varphi}}
\newcommand{\vs}{\ensuremath{\psi}}

\newcommand{\Vp}{\ensuremath{\Psi_{\vp}}}
\newcommand{\Vpk}[1][k]{\ensuremath{\Psi_{\vp_{#1}}}}
\newcommand{\Vs}{\ensuremath{\Psi_{\vs}}}

\newcommand{\vve}{\ensuremath{\widetilde{v}}}

\newcommand{\Kttext}{\ensuremath{W}}

\newcommand{\B}{\ensuremath{B}}
\renewcommand{\S}{\ensuremath{S}}
\newcommand{\Se}{\ensuremath{S_1}}
\newcommand{\St}{\ensuremath{S_2}}
\newcommand{\Sj}{\ensuremath{S_j}}
\newcommand{\Ss}{\ensuremath{S^{*}}}

\newcommand{\Ke}{\ensuremath{K_1}}
\newcommand{\Kt}{\ensuremath{K_2}}
\newcommand{\Kj}[1][j]{\ensuremath{K_{#1}}}

\newcommand{\Pm}[1][m]{\ensuremath{P_{#1}}}
\newcommand{\Qm}[1][m]{\ensuremath{Q_{#1}}}

\newcommand{\rr}{\ensuremath{\rho}}

\newcommand{\N}{\ensuremath{N}}

\newcommand{\Bext}{\ensuremath{\widetilde{\B}}}

\renewcommand{\H}{\ensuremath{H}}
\newcommand{\He}{\ensuremath{H_1}}
\newcommand{\Ht}{\ensuremath{H_2}}
\newcommand{\Hj}[1][j]{\ensuremath{H_{j}}}

\newcommand{\Xj}[1][j]{\ensuremath{X_{#1}}}

\newcommand{\Xm}[1][m]{\ensuremath{\mathcal{X}_{#1}}}
\newcommand{\Xme}{\ensuremath{\mathcal{X}}}

\newcommand{\Oe}{\ensuremath{\Omega_1}}
\newcommand{\Ot}{\ensuremath{\Omega_2}}
\newcommand{\Oj}[1][j]{\ensuremath{\Omega_{#1}}}

\newcommand{\Dd}{\ensuremath{D}}
\newcommand{\BR}{\ensuremath{D}}

\newcommand{\delt}{\ensuremath{\delta}}

\newcommand{\Ote}{\ensuremath{\widetilde{\Omega}_{2}}}

\newcommand{\nhat}{\ensuremath{\widehat{n}}}

\newcommand{\nXY}{\ensuremath{\nabla_{XY}}}

\newcommand{\gd}{\ensuremath{\gamma_{\eta}}}
\newcommand{\Gd}{\ensuremath{\Gamma_{\eta}}}
\newcommand{\gR}{\ensuremath{\gamma_{R}}}
\newcommand{\gz}{\ensuremath{\gamma_{0}}}

\title{Hadamard Type Asymptotics for Eigenvalues of the Neumann Problem for Elliptic Operators}
\author{Vladimir Kozlov}
\author{Johan Thim\footnote{Corresponding author: {\tt johan.thim@liu.se}}}
\affil{\small Department of Mathematics, University of Link\"{o}ping, Link\"{o}ping, Sweden}

\date{\today}

\begin{document}

\maketitle

\begin{abstract}
\noindent
This paper considers how the eigenvalues of the Neumann problem for an elliptic operator depend on the domain. The proximity of two domains is measured in terms of the norm of the difference between the two resolvents corresponding to the reference domain and the perturbed domain, and the size of eigenfunctions outside the intersection of the two domains.  This construction enables the possibility of comparing both nonsmooth domains and domains with different topology.  An abstract framework is presented, where the main result is an asymptotic formula where the remainder is expressed in terms of the proximity quantity described above when this is relatively small.  As an application, we develop a theory for the Laplacian in Lipschitz domains.  In particular, if the domains are assumed to be $C^{1,\alpha}$ regular, an asymptotic result for the eigenvalues is given together with estimates for the remainder, and we also provide an example which demonstrates the sharpness of our obtained result.

\bigskip


\medskip


\end{abstract}


\section{Introduction}
The aim of this article is to describe how the eigenvalues of the Neumann problem for an elliptic
operator depend on the domain. 
A large quantity of studies of the corresponding Dirichlet problem exists in the literature; see, for instance,
Grinfeld~\cite{Grinfeld2010}, Henrot~\cite{Henrot2006}, Kozlov~\cite{Kozlov2013,Kozlov2006}, 
Kozlov and Nazarov~\cite{Kozlov2012}, and references found therein. 
However, less has been written about the Neumann problem. 
In this article, we present a framework for the Neumann problem similar to the one developed for the 
Dirichlet problem in~\cite{Kozlov2013}.

Investigations of how eigenvalues change when the domain is perturbed is a classical problem.
Rayleigh~\cite{Rayleigh1877} studied eigenvalues and domain perturbation in connection with acoustics
as early as in the nineteenth century. 
In the early twentieth century, Hadamard~\cite{Hadamard1968} studied perturbations of domains with smooth boundary,
where the perturbed domain~$\Omega_{\epsilon}$ is represented by~$x_{\nu} = \epsilon h(x')$
where~$x' \in \partial \Omega_0$,~$x_{\nu}$ is the signed distance to the boundary ($x_{\nu} < 0$ for~$x \in \Omega_0$),~$h$ is a smooth
function, and~$\epsilon$ is a small parameter. Hadamard considered the Dirichlet problem, 
but a formula of Hadamard-type for the first nonzero eigenvalue of the Neumann-Laplacian is given by
\[
\Lambda(\Omega_{\epsilon}) = \Lambda(\Omega_0) + 
	\epsilon \int_{\partial \Omega_0} h \bigl( |\nabla \vp|^2 - \Lambda(\Omega_0) \vp^2  \bigr) dS 
	+ o(\epsilon),
\]
where $dS$ is the surface measure on~$\partial \Omega_0$ and~$\vp$ is an eigenfunction corresponding
to~$\Lambda(\Omega_0)$ such that~$\| \vp \|_{L^2(\Omega_0)} = 1$; 
compare with Grinfeld~\cite{Grinfeld2010}.
A study of asymptotics for singular perturbations can be found in, e.g., 
Mel'nyk and Nazarov~\cite{Melnik1997}, Laurain et al. in~\cite{Laurain2011}, Kozlov and Nazarov~\cite{Kozlov2011},
and references found therein.
The problem of domain dependence of eigenvalues is closely related to shape optimization. We refer
to Henrot~\cite{Henrot2006}, and Soko\l{}owski and Zol\'esio~\cite{Sokolowski1992}, and references found therein.

Suppose that~$\Oe$ and~$\Ot$ are domains in~$\R^n$,~$n \geq 2$.
This article considers the spectral problems
\begin{equation}
\label{eq:maineq}
\left\{
\begin{aligned}
& -\Delta  u = \Llm[{}] u   & & \mbox{in } \Oe,\\
&\partial_{\nu}  u  =  0 & & \mbox{on } \partial \Oe
\end{aligned} \right.
\end{equation}
and
\begin{equation}
\label{eq:maineq2}
\left\{
\begin{aligned}
& -\Delta v = \Lum[{}] v  & & \mbox{in } \Ot,\\
&\partial_{\nu} v  =  0 & & \mbox{on } \partial \Ot,
\end{aligned} \right.
\end{equation}
where~$\partial_{\nu}$ is the normal derivative with respect to the outwards normal and if the boundary is nonsmooth,
we consider the corresponding weak formulation of the problem.
Our results are, however, applicable to a wider class of partial differential operators. In particular to uniformly elliptic
operators of second order.

We start the paper with an abstract setting of the problem in a Hilbert space~$H$. We assume that two subspaces~$\He$
and~$\Ht$ are given together with positive definite operators~$\Ke$ and~$\Kt$ acting in~$\He$ and~$\Ht$, respectively.
We assume that~$\Ke$ is a compact operator. We choose an eigenvalue~$\lambda^{-1}$ of~$\Ke$ and denote by~$\Xme \subset \He$
the linear combination of all eigenvectors corresponding to eigenvalues greater than or equal to~$\lambda^{-1}$. The proximity of
the operators~$\Ke$ and~$\Kt$ is measured by a constant~$\eps$ in the inequalities
\[
\| \vp - \S \vp\|^2 \leq \eps \| \vp \|^2 \quad \mbox{for every } \vp \in \Xme
\]     
and
\[
\| (\Kt - \S \Ke \Ss) w \|^2 \leq \eps \| w \|^2 \quad \mbox{for every } w \in \Ht.
\]
Here,~$\S = \St$ is the orthogonal projector from~$H$ into~$\Ht$ and~$\Ss$ is the adjoint operator of~$\S \colon \He \rightarrow \Ht$.
Under these assumption we prove that the operator~$\Kt$ has exactly the same number of eigenvalues in a neighborhood of~$\lambda^{-1}$,
independent of~$\eps$, as the multiplicity of the eigenvalue~$\lambda^{-1}$ of~$\Ke$. This is a consequence of
the continuous dependence of eigenvalues on the domain; see, e.g., Henrot~\cite{Henrot2006}.
Moreover, we present an asymptotic formula for
these eigenvalues where the remainder term is relatively small compared to the leading term. This asymptotic result
improves Theorem~1 in~\cite{Kozlov2006} in two ways. First, we consider~$\He$ and~$\Ht$ as subspaces of a fixed Hilbert space and
can compare operators acting there with the help of orthogonal projectors, which simplifies the conditions of Theorem~1.
Secondly, and perhaps more importantly, the remainder term in our theorem is ``smaller'' with respect to the leading term,
which is not necessarily the case in Theorem~1 from~\cite{Kozlov2006}.

To characterize how close the two domains are,
we will use the Hausdorff distance between the sets~$\Oe$ and~$\Ot$, i.e.,
\begin{equation}
\label{eq:d_def}
d = \max \{ \sup_{x \in \Oe} \inf_{y \in \Ot} |x-y|, \; \sup_{y \in \Ot} \inf_{x \in \Oe} |x-y| \}.
\end{equation}
We do not assume that one domain is a subdomain of the other. It should be noted however, 
that the abstract result presented below permits a more general type of proximity quantity 
for the two domains; see~(\ref{eq:def_eps}) and~(\ref{eq:req_Xm}) in Section~\ref{s:abstract}.

If~$\Oe$ is a~$C^{1,\alpha}$ domain with~$0 < \alpha < 1$ and~$\Ot$ is a Lipschitz perturbation of~$\Oe$
in the sense that the perturbed domain~$\Ot$ can be characterized by a function~$h$
defined on the boundary~$\partial \Oe$ such that every point~$(x',x_{\nu}) \in \partial \Ot$ is 
represented by~$x_{\nu} = h(x')$, where~$(x',0) \in \partial \Oe$ and~$x_{\nu}$ is the signed distance to~$\partial \Oe$
as defined above. 
Moreover, the function~$h$ is assumed to be Lipschitz continuous and satisfy~$|\nabla h| \leq C d^{\alpha}$. 
Then we obtain the following result; see Corollary~\ref{c:holder_boundary}. 

\begin{theorem}
\label{i:t:holder_boundary}
Suppose that~$\Oe$ is a~$C^{1,\alpha}$-domain with~$0 < \alpha < 1$ and~$\Ot$ is as described above, that the
problem in~{\rm(}\ref{eq:maineq}{\rm)} has a discrete spectrum, and that~$m$ is fixed. 
Then there exists a constant~$d_0 > 0$ such that if~$d \leq d_0$, then 
\begin{equation}
\label{i:eq:holder_main_boundary}
\begin{aligned}
\Lum[k] - \Llm[m] = {} & 
\kk + O(d^{1+\alpha})
\end{aligned}
\end{equation}
for every~$k = 1,2,\ldots,\Jm[m]$, where~$\Jm[m]$ is the dimension of the eigenspace corresponding to~$\Llm[m]$.
Here~$\kk[{}] = \kk$ is an eigenvalue of the problem
\begin{equation}
\label{eq:T2_tk_intro}
\kk[{}] \ip{\vp}{\vs}{1} = 
\int_{\partial \Oe} h(x') \bigl( 
	\nabla \vp \cdot \nabla \vs 
	- \Llm \vp \vs
	\bigr) dS(x')
\quad \mbox{for all } \vs \in \Xj[m], 
\end{equation}
where~$\vp \in \Xj[m]$. Moreover, $\kk[1],\kk[2],\ldots,\kk[\Jm]$ in~{\rm(}\ref{i:eq:holder_main_boundary}{\rm)}
run through all eigenvalues of~{\rm(}\ref{eq:T2_tk_intro}{\rm)} counting their multiplicities.
\end{theorem}

\noindent
Observe that~(\ref{eq:T2_tk_intro}) can be phrased as a spectral problem on the Hilbert space~$\Xj[m]$ by using the 
Riesz representation theorem of the operator on the right-hand side.

In Section~\ref{s:ex}, we consider a specific example of a Lipschitz perturbation of a cylinder in two dimensions. 
We prove that if~$\eta:  \R \rightarrow \R$ is a periodic nonnegative Lipschitz continuous 
function with period~$1$, 
and~$\Oe \subset \R^2$ is defined by~$0 < x < T$ and~$0 < y < R$, where~$R$ and~$T$
are constants, and the subdomain~$\Ot \subset \Oe$ is defined by~$0 < x < T$ and~$\delt \eta(x / \delt) < y < R$
for a small parameter~$\delt$, then
\begin{equation}
\label{intro:ex:main}
\Lum[k] - \Llm[m] = \kk + O(\delt^2)
\quad \mbox{for every } k = 1,2,\ldots,\Jm[m], 
\end{equation}
where~$\Jm[m]$ is the dimension of the eigenspace corresponding to~$\Llm[m]$.
Here, $\kk[{}] = \kk$ is an eigenvalue of the problem
\begin{equation}
\label{eq:ex:T2_tk_intro}
\begin{aligned}
\kk[{}] \ip{\vp}{\vs}{1} = {} &
\delt \eta_0 \int_0^T \bigl( \nabla \vp(x,0) \cdot \nabla \vs(x,0) 
 - \Llm \vp(x,0) \vs(x,0) \bigr) \, dx\\
&+
\delt \eta_1 \int_0^T \nabla \vp(x,0) \cdot \nabla \vs(x,0)  \, dx
\end{aligned}
\end{equation}
for all $\vs \in \Xj[m]$, where~$\vp \in \Xj[m]$ and
\[
\eta_0 = \int_0^1 \eta(X) \, dX \qquad \mbox{and} \qquad 
\eta_1 = \int_0^1 V(X,\eta(X)) \eta'(X) \, dX.
\]
The function~$V$ is the solution to~$-\Delta V = 0$ for~$0 < X < 1$ and~$Y > \eta(X)$ with the boundary 
condition~$\partial_{\nu} V(X,\eta(X)) = \eta'(X)(1+(\eta'(X))^2)^{-1/2}$ on~$Y = \eta(X)$ and periodic boundary conditions
on the remaining boundary. The constant~$\eta_1$ is not zero if~$\eta$ is not identically constant.
Observe that the first term in the right-hand side of~(\ref{eq:ex:T2_tk_intro}) coincides with  the right-hand
side of~(\ref{eq:T2_tk_intro}) up to order~$O(\delt^2)$. This proves that the condition~$\alpha > 0$ is sharp 
in Theorem~\ref{i:t:holder_boundary}.

In Corollary~\ref{c:lipschitz}, we obtain as a consequence of the methods developed 
that eigenvalues satisfy the following estimate for (uniformly) Lipschitz perturbations.
There exists a constant~$C$, independent of~$d$, such that
$|\Lum[k]- \Llm[m]| \leq C d $
for every~$k=1,2,\ldots,\Jm$; see Corollary~\ref{c:lipschitz}.
This estimate can be compared to results presented in, e.g., 
Burenkov and Davies~\cite{Burenkov2002} in the case when~$\Ot \subset \Oe$.

\section{Abstract Setting: Perturbation of Eigenvalues}
\label{s:abstract}
The fact that zero is an eigenvalue for the problems in~(\ref{eq:maineq}) and~(\ref{eq:maineq2}) is trivial, and to avoid technicalities
due to this, we will consider the operator~$1 -\Delta$. A number~$\lm[{}]$ is an eigenvalue
of the operator~$1 - \Delta$ if and only if~$\lm[{}] - 1$ is an eigenvalue of~$-\Delta$.
Let~$\Llm[k] = \lm[k]-1$,~$k=1,2,\ldots$, be the eigenvalues of~(\ref{eq:maineq}) 
enumerated according to
$
0 < \lm[1] < \lm[2] < \cdots .
$
We assume here that~$\Oe$ is connected.
Similarly, we let~$\Lum[k] = \um[{}] - 1$ be the eigenvalues of~(\ref{eq:maineq2}).
The subset~$\Xj[k]$ of~$\He$ is the eigenspace corresponding to eigenvalue~$\Llm[k]$, with the dimension denoted 
by~$\Jm[k] = \mbox{dim}(\Xj[k])$. Observe that~$\Xj[k]$ is also the eigenspace for the eigenvalue~$\lm[k]$ 
of the to~(\ref{eq:maineq}) corresponding problem for~$1 - \Delta$.

We proceed by introducing an abstract setting for the problems in~(\ref{eq:maineq}) and~(\ref{eq:maineq2}).
Suppose that~$\He$ and~$\Ht$ are infinite dimensional subspaces of a Hilbert space~$H$.
Let the operators~$\Kj \colon \Hj \rightarrow \Hj$ be
positive definite and self-adjoint for~$j=1,2$. Furthermore, let~$\Ke$ be compact.
We consider the spectral problems
\begin{equation}
\label{eq:eig_k1}
\Ke \vp = \lambda^{-1} \vp, \quad \vp \in \He,
\end{equation}
and
\begin{equation}
\label{eq:eig_k2}
\Kt U = \mu^{-1} U, \quad U \in \Ht,
\end{equation}
and denote by~$\lm[k]^{-1}$ for~$k=1,2,\ldots$ the eigenvalues of~$\Ke$.
Let~$\Xj[k] \subset \He$ be the eigenspace corresponding to eigenvalue~$\lm[k]^{-1}$. Moreover, we denote
the dimension of~$\Xj[k]$ by~$\Jm[k]$ and define~$\Xm = \Xj[1] + \Xj[2] + \cdots \Xj[m]$,
where~$m \geq 1$ is any integer. 
In this article we study eigenvalues of~(\ref{eq:eig_k2}) located in a neighborhood of~$\lm$,
where~$m$ is fixed.

In order to define the proximity of the operators~$\Ke$ and~$\Kt$,
which are defined on different spaces, we introduce the orthogonal projectors~$\Se \colon \H \rightarrow \He$ 
and~$\St \colon \H \rightarrow \Ht$. To simplify the notation, we also introduce
the operator~$\S$ as the restriction of~$\St$ to~$\He$. Thus~$\S$ maps~$\He$
into~$\Ht$ and its adjoint operator~$\Ss \colon \Ht \rightarrow \He$ is given by~$\Ss = \Se \St$.

We introduce a quantity~$\eps > 0$ as a constant in the inequalities
\begin{equation}
\label{eq:def_eps}
\ds \| (\Kt - \S \Ke \Ss) w \|^2 \leq \eps \| w \|^{2} \quad \mbox{for every } w \in \Ht
\end{equation}
and
\begin{equation}
\label{eq:req_Xm}
\| \vp - \S \vp \|^2 \leq \eps \| \vp \|^2
\quad \mbox{for every } \vp \in \Xm.
\end{equation}
The parameter~$\eps$ is the measure we use to describe the proximity of the spaces~$\He$ and~$\Ht$
and the operators~$\Ke$ and~$\Kt$.
In the following analysis, an important role is played by
the operator~$B \colon \He \rightarrow \Ht$ defined as
\[
\B = \Kt \S - \S \Ke.
\]

\begin{remark}
A common way to compare the proximity of domains in shape optimization is the parameter~$\sigma$ in
\begin{equation}
\label{eq:def_eps_sym}
\| (\Kt \St - \Ke \Se) w \|^2 \leq \sigma \| w \|^2
\quad \mbox{for every } w \in \H. 
\end{equation}
Let us show that~$\eps$ can be chosen as
\[
\eps = \sigma \max \{ 1, \, 4 \sum_{k=1}^m \lm[k]^2 \}.
\]
The fact that~(\ref{eq:def_eps}) holds can be verified directly. 
To verify that~(\ref{eq:req_Xm}) holds, let~$\vp \in \Xm$.
Then~$\vp = \sum_{k=1}^m c_k \vp_k$, where~$\vp_k \in \Xj[k]$ are orthonormal and~$c_k$
are constants. Thus, 
\[
\begin{aligned}
\| \vp - \S \vp \| &\leq \sum_{k=1}^m |c_k \lm[k]| \| \Ke \vp_k - \S \Ke \vp_k \| \\
&\leq \sum_{k=1}^m |c_k\lm[k]| ( \| \Ke \vp_k - \Kt \St \vp_k\| + \| \Kt \St \vp_k - \Ke \vp_k\|)\\
&\leq 2 \sigma^{1/2} \sum_{k=1}^{m} |c_k\lm[k]|,
\end{aligned}
\]
which implies that
\[
\| \vp - \S \vp\|^2 \leq 4\sigma \biggl( \sum_{k=1}^m \lm[k]^2 \biggr) \| \vp \|^2.
\]
\end{remark}


\section{Main Results}
Let~$\Pm$ be the orthogonal projection of~$\He$ onto~$\S \Xj[m]$. 
We now state results about the stability of eigenvalues and eigenvectors
depending on the parameter~$\eps$. The first lemma is a consequence of
the continuous dependence of eigenvalues on the domain; see, for instance,
Kato~\cite{Kato1966} (Sections~IV.3 and~V.3) or Henrot~\cite{Henrot2006} and references therein.

\begin{lemma}
\label{l:mainprop}
There exists positive constants~$\eps_0$,~$c$, and~$C$, depending
on the eigenvalues~$\lm[1],\ldots,\lm[m+1]$, such that, for~$\eps \leq \eps_0$, the
following assertions are valid\/{\rm:}
\begin{enumerate}[label={\rm(\roman{*})}, ref={\rm(\roman{*})}]
\item The operator~$\Kt$ has precisely~$\Jm$ eigenvalues in
$\bigl(\lm[m+1]^{-1} + c\eps^{1/2}, \, \lm[m-1]^{-1} - c\eps^{1/2} \bigr)$
and all of them are located in
$\bigl(\lm[m]^{-1} - c\eps^{1/2}, \, \lm[m]^{-1} + c\eps^{1/2} \bigr)$.
\item
If~$\um[{}]^{-1}$ is an eigenvalue of~{\rm(}\ref{eq:eig_k2}{\rm)} from the interval 
$\bigl(\lm[m]^{-1} - c\eps^{1/2}, \, \lm[m]^{-1} + c\eps^{1/2} \bigr)$
and~$U$ is a corresponding eigenfunction, then
\[
\| U - \Pm U \| \leq C \eps^{1/2} \| U \|.
\]
\end{enumerate}
\end{lemma}

\noindent We denote by~$\um[k]^{-1}$ for~$k=1,2,\ldots,\Jm[m]$, the eigenvalues of
the spectral problem~(\ref{eq:eig_k2}) located in the interval
$(\lm[m]^{-1} - c\eps^{1/2}, \, \lm[m]^{-1} + c\eps^{1/2})$, where~$c$ is the same constant
as in Lemma~\ref{l:mainprop}.
The quantity~$\rr$ is defined by
\begin{equation}
\label{eq:def_rho}
\rr = \sup_{\vp \in \Xj[m], \; \| \vp \| = 1} 
	\bigl( \lm \|\Kt \B \vp \|^2 + \eps \lm \| \B \vp \|^2 \bigr).
\end{equation}

\begin{theorem}
\label{t:asymp}
The following asymptotic formula holds\/{\rm:}
\begin{equation}
\label{eq:T2_um}
\um[k]^{-1} = \lm[m]^{-1} + \tk + O(\rr + |\tk|\eps) \quad
\mbox{for every } k = 1,2,\ldots,\Jm[m],
\end{equation}
where~$\tau = \tk$ is an eigenvalue of the problem
\begin{equation}
\label{eq:T2_tk}
\tau \ip{\S \vp}{\S \vs}{2} = \lm \ip{\B\vp}{\B\vs}{2} + \ip{\B \vp}{\S \vs}.
\quad \mbox{for all } \vs \in \Xj[m],
\end{equation}
where~$\vp \in \Xj[m]$. Moreover, $\tk[1],\tk[2],\ldots,\tk[m]$ in~{\rm(}\ref{eq:T2_um}{\rm)}
run through all eigenvalues of~{\rm(}\ref{eq:T2_tk}{\rm)} counting their multiplicities.
\end{theorem}

\noindent 
In applications, the term~$\|\Kt \B \vp \|$ is typically significantly smaller than~$\max |\tk|$; 
see, e.g., Lemma~\ref{l:est_KtB}. This implies that~$\rr$ is small compared to~$\tk$ for every~$k$.  
Note also that the right-hand side of~(\ref{eq:T2_tk}) can be expressed 
more compactly as~$\lm \ip{\B\vp}{\Kt\S\vs}{2}$.

The asymptotic formula in~(\ref{eq:T2_um}) has similarities to the one presented in Kozlov~\cite{Kozlov2006}.
The main difference is how the remainder term is constructed; in Theorem~\ref{t:asymp},~$\rr$ is typically small compared to the main terms
above. However, the same is not necessarily true in~\cite{Kozlov2006}.

\section{Proof of Lemma~\ref{l:mainprop}}
\label{s:proof_prop}
The following properties hold.

\begin{enumerate}[
label=(\Roman{*}),
ref=(\Roman{*}), 
]
\item $\ds (1 - \eps) \| \vp \|^2 \leq \| \S \vp \|^2 \leq \| \vp \|^2$ for every~$\vp \in \Xm$.
\label{prop:est_vp_svp}
\label{eq:est_vp_svp}
\item 
There exists a positive constant~$C$, depending on the operator norm of~$\Ke$, 
such that
\begin{equation}
\label{eq:est_B}
\| \B \vp \| \leq C \eps^{1/2} \| \vp \| \quad \mbox{for all } \vp \in \Xm.
\end{equation}
\label{prop:Bbounded}
\item $\ds \ip{\Kt w}{w}{2} \leq \ip{\Ke\Ss w}{\Ss w}{1} + \eps^{1/2} \| w \|^2 \quad$for all~$w \in \Ht$.
\label{prop:ip}
\label{last-item2}
\end{enumerate}
The inequality in~\ref{prop:est_vp_svp} follows from
\[
\| \S \vp \|^2  \geq \| \vp \|^2 - \| \vp - \S \vp\|^2 
\geq
(1 - \eps) \| \vp \|^2.
\]
To prove~\ref{prop:Bbounded}, suppose that~$\vp \in \Xm$. 
Then
\[
\begin{aligned}
\| \B \vp \| &\leq \| \B \Ss \S \vp \| + \| \B (\vp - \Ss \S \vp) \|\\
&\leq \| (\Kt - \S \Ke \Ss) \S \vp \| + \| \S \Ke (\vp - \Ss \S \vp) \|\\
&\leq \eps^{1/2} \| \S \vp \| + C \| \vp - \Ss \S \vp \|\\
&\leq C \eps^{1/2} \| \vp \|,
\end{aligned}
\]
where we used~(\ref{eq:def_eps}), (\ref{eq:req_Xm}), and the fact that~$\S \Ss \S \vp = \S \vp$.
The property in~\ref{prop:ip} follows from the fact that
\[
\begin{aligned}
\ip{\Kt w}{w}{2} - \ip{\Ke \Ss w}{\Ss w}{1}
&= 
\ip{(\Kt - \S \Ke \Ss) w}{w}{2} \\
&\leq \| (\Kt - \S \Ke \Ss) w\| \|w\|
 \leq \eps^{1/2} \| w \|^2.
\end{aligned}
\]
The arguments in Section~3.2, 3.3, and~3.4 of~\cite{Kozlov2013}, are now valid with small modifications.
Specifically, we substitute~$\S$ for the operator~$S_2$ in these sections, 
and replace inequality~(32) by~\ref{prop:est_vp_svp}.
Furthermore, the proof of inequality~(34) is analogous, 
inequality~(36) is replaced by~\ref{prop:Bbounded}, and finally, inequality~(39) is replaced by~\ref{prop:ip}.
This completes the proof of Lemma~\ref{l:mainprop}.

\section{Proof of Theorem~\ref{t:asymp}}
The proof of Theorem~\ref{t:asymp} mirrors that of the corresponding theorem in Kozlov~\cite{Kozlov2013}. 
Equation~(\ref{eq:Vp}) below corresponds to~(7) in~\cite{Kozlov2013}, but in this case we have
the explicit solution given in~(\ref{eq:solution_Vp}). In Sections~\ref{s:QmBvBs}--\ref{s:theorem_asymp_proof}, 
we provide results similar to the ones found in Section~4 of~\cite{Kozlov2013}. 

Let~$\Qm = I - \Pm$, where~$I$ is the identity operator on~$\Ht$, and suppose henceforth that~$\vp$
and~$\vs$ belong to~$\Xj[m]$.
To simplify the notation, define 
\begin{equation}
\label{eq:solution_Vp}
\Vp = -\lm \B \vp \quad \mbox{for any } \vp \in \Xj[m].
\end{equation}
Then~$\Vp$ solves the equation
\begin{equation}
\label{eq:Vp}
\ip{\Vp}{w}{2} = \ip{\vp}{w}{2} - \lm\ip{\S\vp}{\Kt w}{2}
\quad \mbox{for every } w \in \Ht.
\end{equation}
To verify~(\ref{eq:solution_Vp}), suppose that~$w \in \Ht$. Then
\[
\begin{aligned}
\ip{\B\vp}{w}{2} &= \ip{\Kt \S \vp}{w}{2} - \ip{\S \Ke \vp}{w}{2}\\
&= -\lm^{-1} \bigl( \ip{\S\vp}{w}{2} - \lm \ip{\S \vp}{\Kt w}{2} \bigr)\\
&= -\lm^{-1} \ip{\Vp}{w}{2}.
\end{aligned}
\]

\subsection{Representation of \boldmath{$\ip{\Qm \B \vp}{\B \vs}{2}$}}
\label{s:QmBvBs}
From~(\ref{eq:solution_Vp}) it follows that
\[
\ip{\Qm \B \vp}{\B \vs}{2} = 
\lm^{-2} \bigl( 
\ip{\Vp}{\Vs}{2} - \ip{\Pm \Vp}{\Vs}{2}
\bigr).
\]
Let~$\{\Upsilon_k\}_{k=1}^{\Jm}$ be an ON-basis in~$\S \Xj[m]$. Then, for each~$k=1,\ldots,\Jm$, there
exists~$\vp_k \in \Xj[m]$ such that~$\Upsilon_k = \S \vp_k$. Thus,
\begin{equation}
\label{eq:PmVpVs}
\ip{\Pm \Vp}{\Vs}{2} = \sum_{k=1}^{\Jm} \ip{\Vp}{\S \vp_k}{2}\ip{\S \vp_k}{\Vs}{2}.
\end{equation}
From~(\ref{eq:solution_Vp}) and~\ref{prop:Bbounded}, it is clear that
\[
| \ip{\Vp}{\S \vp_k}{2} | =
	\lm |\ip{\B \vp}{\S \vp_k}{2}| \leq
	\lm \|\B \vp \| \| \S \vp_k \| \leq
	C \eps^{1/2} \| \vp \|
\]
for~$k=1,\ldots,\Jm$.
Moreover, letting~$w = \S \vp_k$ in~(\ref{eq:Vp}) proves that
\begin{equation}
\label{eq:VpSvpk}
\begin{aligned}
\ip{\Vp}{\S\vp_k}{2} &=
\lm \ip{\Vp}{\Kt\S\vp_k}{2} + \ip{\Vp}{\Vpk}{2}\\ 
&= \lm \ip{\Kt \Vp}{\S\vp_k}{2} - \lm \ip{\Vp}{\B \vp_k}{2},
\end{aligned}
\end{equation}
from which it follows together with~\ref{eq:est_vp_svp} that
\begin{equation}
\label{eq:VpSvpk_est}
\begin{aligned}
|\ip{\Vp}{\S\vp_k}{2}| &\leq
\lm \bigl( \| \Kt \Vp \| \| \S \vp_k \| + \| \Vp\| \|\B \vp_k \| \bigr)\\
&\leq
C \bigl( \| \Kt \B \vp \|  + \eps^{1/2} \| \B \vp \| \bigr)
\end{aligned}
\end{equation}
Analogously,
\[
\ip{\Vs}{\S\vp_k}{2} = \lm \ip{\Kt \Vs}{\S \vp_k}{2} + \ip{\Vs}{\Vpk}{2},
\]
and thus
\begin{equation}
\label{eq:SvpkVs}
| \ip{\S \vp_k}{\Vs}{2} | \leq C \bigl( \| \Kt \B \vs \|  + \eps^{1/2} \| \B \vs \| \bigr).
\end{equation}
Now, the identity in (\ref{eq:PmVpVs}), and the estimates in~(\ref{eq:VpSvpk_est}) and~(\ref{eq:SvpkVs}), imply that
\begin{equation}
\label{eq:est_PmBvpBvs}
| \ip{\Pm\Vp}{\Vs}{2} | = \lm^2 |\ip{\Pm\B\vp}{\B\vs}{2}| \leq C \bigl( \rr(\vp) + \rr(\vs) \bigr),
\end{equation}
where
\[
\rr(\vp) = 
 \lm \bigl( \|\Kt \B \vp \|^2 + \eps \| \B \vp \|^2 \bigr),
\quad \vp \in \Xj[m].
\]

\subsection{Estimate of~\boldmath{$\ip{\Kt \Qm \B \vp}{\Qm \B \vs}{2}$}}
Since~$\Pm + \Qm = I$, it is clear that
\[
\begin{aligned}
\ip{\Kt \Qm \B \vp}{\Qm \B \vs}{2} = {} &
\ip{\Kt \B \vp}{\Qm \B \vs}{2} -
\ip{\Kt \Pm \B \vp}{\Qm \B \vs}{2}\\
= {} &
\ip{\Kt \B \vp}{\Qm \B \vs}{2} -
\ip{\Kt \Pm \B \vp}{\B \vs}{2}\\
&+ \ip{\Kt \Pm \B \vp}{\Pm \B \vs}{2}.
\end{aligned}
\]
Now,
\[
|\ip{\Kt \B \vp}{\Qm  \B \vs}{2}|
\leq
\| \Kt \B \vp \| \| \Qm \B \vs\|
\leq
C \eps^{1/2} \| \vs \| \| \Kt \B \vp\|.
\]
Similarly, 
\[
|\ip{\Kt \Pm \B \vp}{\B \vs}{2}|
=
|\ip{\Pm \B \vp}{\Kt \B \vs}{2}|
\leq
C \eps^{1/2} \| \vp \| \| \Kt \B \vs\|.
\]
As in Section~\ref{s:QmBvBs},
let~$\{\Upsilon_k\}_{k=1}^{\Jm}$ be an ON-basis in~$\S \Xj[m]$. Then
there exists eigenfunctions~$\vp_k \in \Xj[m]$ such that~$\Upsilon_k = \S \vp_k$
for every~$k=1,\ldots,\Jm$.
Thus,
\begin{equation}
\label{eq:KtPmBvp}
\ip{\Kt \Pm \B \vp}{\Pm \B \vs}{2} = \sum_{k=1}^{\Jm} \ip{\Pm \B \vp}{\S \vp_k}{2}\ip{\Kt \S \vp_k}{\Pm \B \vs}{2}.
\end{equation}
Using~(\ref{eq:solution_Vp}) and~(\ref{eq:VpSvpk_est}), it is clear that
\[
|\ip{\Pm \B \vp}{\S \vp_k}{2}| 
\leq 
|\ip{\B \vp}{\S \vp_k}{2}| 
\leq
C \bigl( \| \Kt \B \vp \| + \eps^{1/2} \| \B \vp \| \bigr).
\]
Furthermore,~(\ref{eq:solution_Vp}) and~(\ref{eq:VpSvpk_est}), with~$\vp_l$ in the place of~$\vp_k$
and~$\vs$ replaced by~$\vp$, proves that
\[
\begin{aligned}
|\ip{\Kt \S \vp_k}{\Pm \B \vs}{2}|
&=
\biggl| \sum_{l=1}^{\Jm} \ip{\Kt \S \vp_k}{\S \vp_l}{2} \ip{\Pm \B \vs}{\S \vp_l}{2} \biggr|\\
&\leq
C \sum_{l=1}^{\Jm} |\ip{\Pm \B \vs}{\S \vp_l}{2}|\\
&\leq C \bigl( \| \Kt \B \vs \| + \eps^{1/2} \| \B \vs \| \bigr).
\end{aligned}
\]
Thus,
\[
\bigl( \| \Kt \B \vp \| + \eps^{1/2} \| \B \vp \| \bigr)
		\bigl( \| \Kt \B \vs \| + \eps^{1/2} \| \B \vs \| \bigr)
\leq C \bigl( \rr(\vp) + \rr(\vs) \bigr).
\]
Finally, we obtain that
\begin{equation}
\label{eq:est_KtQmBvpQmBvs}
|\ip{\Kt \Qm \B \vp}{\Qm \B \vs}{2}|
\leq
C \bigl( \rr(\vp) + \rr(\vs) \bigr).
\end{equation}

\subsection{Proof of Theorem~\ref{t:asymp}}
\label{s:theorem_asymp_proof}
Analogously with the argument used in Kozlov~\cite{Kozlov2013}, it is possible to reduce the spectral
problem~(\ref{eq:eig_k2}) to a finite dimensional situation using the projectors~$\Pm$ and~$\Qm$:
\begin{equation}
\label{eq:k2_t2}
(\um[{}]^{-1} - \Kt)(\S \vp + w) = 0,
\end{equation}
where~$\vp \in \Xj[m]$ and~$w \in \Qm \Ht$.
Indeed, proceeding accordingly with Section~4.1 in~\cite{Kozlov2013}, we obtain that
\begin{equation}
\label{eq:aa}
\tke \ip{\S \vp}{\S \vs}{2}
- \ip{\B \vp}{\S \vs}{2} 
- \um[{}] \ip{\Qm\B\vp}{\B\vs}{2}
- \ip{L(\um[{}])\B\vp}{\B\vs}{2} = 0,
\end{equation}
where $L(\um[{}]) = \um[{}] \Qm \Kt \Qm \bigl( \um[{}]^{-1} - \Qm \Kt \Qm \bigr)^{-1} \Qm$
and~$\tke = \um[{}]^{-1} - \lm^{-1}$.
We assume that~$|\tke| \leq \eps^{1/2}$.
Moreover, the operator~$\bigl( \um[{}]^{-1} - \Qm \Kt \Qm \bigr)^{-1}$ is bounded
from~$\Qm \Ht$ into~$\Qm \Ht$:
\[
\| \bigl( \um[{}]^{-1} - \Qm \Kt \Qm \bigr)^{-1} w \|_{\Qm \Ht}
\leq C \| w \|_{\Qm \Ht}
\quad \mbox{for every } w \in \Ht.
\]
Hence,
\[
|\ip{L(\um[{}])\B \vp}{\B \vs}{2}| \leq C \ip{\Kt \Qm \B \vs}{\Qm \B \vs}{2}.
\]
It follows from the identity~$\um[{}]^{-1} = \lm^{-1} + \tke$ that
\[
\um[{}] \ip{\Qm\B\vp}{\B\vs}{2} = \lm\ip{\B \vp}{\B \vs}{2} - b_2(\vp,\vs),
\]
where
\[
b_2(\vp,\vs) = \frac{\lm\tke}{\tke + \lm^{-1}} \ip{\Qm\B\vp}{\B\vs}{2} + \um[{}]\ip{\Pm \B \vp}{\B \vs}{2}.
\]
Then
\begin{equation}
\label{eq:est_b_2}
|b_2(\vp,\vs)| \leq C |\tke| \eps + C \bigl( \rr(\vp) + \rr(\vs) \bigr).
\end{equation}
Put~$b(\vp,\vs) = \ip{L(\um[{}])\B\vp}{\B\vs}{2} + b_2(\vp,\vs)$. Then
\begin{equation}
\label{eq:aaa}
\tke
\ip{\S \vp}{\S \vs}{2} = \lm \ip{\B \vp}{\Kt \S \vs}{2} + b(\vp,\vs),
\end{equation}
where~$b(\vp,\vs)$ satisfies
\begin{equation}
\label{eq:est_b}
|b(\vp,\vs)| \leq C \bigl( \rr(\vp) + \rr(\vs) + |\tke| \eps \bigr)
\end{equation}
according to~(\ref{eq:est_b_2}) and~(\ref{eq:est_KtQmBvpQmBvs}).

Suppose that~$j = 1,\ldots,\Jm$.
Let~$U_j \in \Ht$ be an eigenfunction of~$\Kt$ corresponding to the eigenvalue~$\um[j]^{-1}$.
Then there exists~$V_j \in \Xj[m]$ satisfying~$\Pm U_j = \S V_j$. 
By~$\tke_j$ we denote an eigenvalue of~(\ref{eq:aaa}) with eigenfunction~$\vp = V_j$.
Suppose also
that~$\tk$ is an eigenvalue of~(\ref{eq:T2_tk}) and~$\Phi_j \in \Xj[m]$ a corresponding eigenfunction.
Analogously with Section~4.5 in Kozlov~\cite{Kozlov2013}, we may assume that there exists a constant~$c_{*} > 0$
such that
\begin{equation}
\label{eq:VjPj}
\ip{\S V_j}{\S \Phi_j}{2} \geq c_{*}
\end{equation}
after possible rearrangement of the eigenfunctions~$\Phi_j$ spanning~$\Xj[m]$.

Choosing~$\vp = \Phi_j$ and~$\vs = V_j$ in equation~(\ref{eq:T2_tk}),
and~$\vp = V_j$ and~$\vs = \Phi_j$ in equation~(\ref{eq:aaa}), and then
subtracting~(\ref{eq:T2_tk}) from~(\ref{eq:aaa}), we obtain that
\[
(\tke_j - \tk[j])\ip{\S V_j}{\S \Phi_j}{2} = \lm \bigl( \ip{\B V_j}{\Kt\S\Phi_j}{2} - \ip{\B\Phi_j}{\Kt\S V_j}{2} \bigr)
+ b(V_j,\Psi_j).
\]
The fact that~$\Kt$ is self-adjoint, that~$\Phi_j$ and~$V_j$ belong to~$\Xj[m]$, and the definition of~$B$, imply that
\[
\ip{\B V_j}{\Kt\S\Phi_j}{2} - \ip{\B\Phi_j}{\Kt\S V_j}{2}
=
\lm^{-1} \bigl( \ip{\S \Phi_j}{\Kt \S V_j}{2} - \ip{\S V_j}{\Kt \S \Phi_j}{2} \bigr) = 0.
\]
Hence,
\[
(\tke_j - \tk[j])\ip{\S V_j}{\S \Phi_j}{2} = b(V_j,\Psi_j),
\]
from which it follows from~(\ref{eq:est_b}) and~(\ref{eq:VjPj}) that
\[
|\tke_j - \tk[j]| \leq C \bigl( \rr(V_j) + \rr(\Psi_j) + |\tke_j|\eps \bigr).
\]
Taking the supremum over~$V_j$ and~$\Psi_j$ in~$\Xj[m]$ with~$\| V_j \| = \| \Psi_j \| = 1$, 
we obtain that
\[
|\tke_j - \tk[j]| \leq C \bigl( \rr + |\tke_j|\eps \bigr),
\]
where
\[
\rr = \sup_{\stackrel{\vp \in \Xj[m]}{\| \vp \| = 1}} \rr(\vp)
=
\lm \sup_{\stackrel{\vp \in \Xj[m]}{\| \vp \| = 1}} \bigl( \|\Kt \B \vp \|^2 + \eps \| \B \vp \|^2 \bigr).
\]
This also implies that
\[
|\tke_j - \tk[j]| \leq C \bigl( \rr + |\tk[j]|\eps \bigr),
\]


\section{Applications}
In this section we consider the Neumann problem for the operator~$1 - \Delta$ in different domains.
Let~$\Oe$ and~$\Ot$ be two domains in~$\R^n$ with nonempty intersection.
We put~$\H = L^2(\R^n)$ and~$\Hj = L^2(\Oj)$ for~$j = 1,2$. Functions in~$\Hj$ are extended to~$\R^n$ by zero
outside of~$\Oj$.
Observe that we do not require that one subdomain~$\Oj$ is a subset of the other.
For~$f \in L^2(\Oj)$,
the weak solution to the Neumann problem~$(1 - \Delta) W_j = f$ in~$\Oj$ and~$\partial_{\nu} W_j = 0$ on~$\partial \Oj$
for~$j=1,2$ satisfies 
\begin{equation}
\label{eq:weak_neumann}
\int_{\Oj} ( \nabla W_j \nabla v  + W_j v ) \, dx = \int_{\Oj} f  v \, dx
\quad \mbox{for every } v \in H^1(\Oj).
\end{equation}
It follows from~(\ref{eq:weak_neumann}) with~$v = W_j$ and the Cauchy-Schwarz inequality that
\[
\| \nabla W_j \|_{L^2(\Oj)} + \| W_j \|_{L^2(\Oj)} \leq \| f \|_{L^2(\Oj)}.
\]
We let~$\Kj$ for~$j=1,2$ be defined on~$L^2(\Oj)$ as the solution operators corresponding to the domains~$\Oj$,
i.e.,~$\Kj f = W_j$.
Then~$\Kj$ maps~$L^2(\Oj)$ into the Sobolev space~$H^1(\Oj)$, and
\[
\| \Kj u \|_{H^1(\Oj)} \leq C \| u \|_{L^2(\Oj)}.
\]
Moreover,~$(1 - \Delta) \Kj u = u$ and~$\partial_{\nu} \Kj u = 0$ on~$\partial \Oj$ in the weak sense.
The operators~$\Kj$ are also self-adjoint and positive definite, and if~$\Oj$ are, e.g., Lipschitz, also compact.

To characterize how close the two domains are,
we will use the Hausdorff distance~$d$ between the sets~$\Oe$ and~$\Ot$ given in~(\ref{eq:d_def}).

\subsection{Perturbations of Lipschitz- and~$C^{1,\alpha}$-Domains}
\label{s:domains}
We now consider two cases of regularity of the boundaries~$\partial \Oj$, namely~$C^{1,\alpha}$
and Lipschitz boundaries. Let us first consider the Lipschitz case. 
Let~$\Oe$ be the reference domain which will be fixed throughout. 
Then there exists a positive constant~$M$ such that the boundary~$\partial \Oe$ can be covered by
a finite number of balls~$B_k$,~$k=1,2,\ldots,N$, where there exists orthogonal coordinate systems in which
\[
B_k \cap \Omega_1= B_k \cap \{ y = (y', y_n) \sa y_n > h_k^{(1)}(y')\}
\]
where the center of~$B_k$ is at the origin and~$h_k^{(1)}$ 
are Lipschitz functions, i.e., 
\[
|h_k^{(1)}(y') - h_k^{(1)}(x')| \leq M |y' - x'|,
\]
such that~$h_k^{(1)}(0) = 0$.
We assume that~$\Ot$ belongs to the class of domains where~$\Ot$ is close to~$\Oe$ in the sense that~$\Ot$ can be described by
\[
B_k \cap \Omega_2= B_k \cap \{ y = (y', y_n) \sa y_n > h_k^{(2)}(y')\},
\]
where~$h_k^{(2)}$ are also Lipschitz continuous with Lipschitz constant~$M$.
Clearly all such domains belong to a ball~$\Dd$ of sufficiently large radius depending only on~$M$ and~$B_1,B_2,\ldots,B_N$. 
Note also that~$\Oe \cap \Ot$ is a Lipschitz domain of this type and that we can use the same covering and Lipschitz constant.

\begin{remark}
Observe that~$d$ is comparable to
\[
\widehat{d} = \max_{k=1,2,\ldots,N} \sup \{ | h_k^{(1)}(y') - h_k^{(2)}(y')| \sa y = (y',y_n) \in B_k \cap \partial \Oe \} 
\]
in the sense that there exists positive constants~$c_1$ and~$c_2$ depending only on~$M$ and~$B_k$,~$k=1,2,\ldots,N$, such that
$ c_1 \widehat{d} \leq d \leq c_2 \widehat{d}$.
\end{remark}
\noindent
The case of a~$C^{1,\alpha}$ domain is defined in the same manner, with the additional assumptions
that~$h_k^{(1)}$ are~$C^{1,\alpha}$-functions such 
that
\[
h_k^{(1)}(0) = \partial_{x_i} h_k^{(1)} = 0, \quad i=1,2,\ldots,n-1.
\] 
Moreover, we suppose that
\begin{equation}
\label{eq:grad_hk}
| \nabla ( h_k^{(1)} - h_k^{(2)}) | \leq C d^{\alpha}.
\end{equation}
Note that~$h_k^{(2)}$ are only assumed to be Lipschitz continuous and satisfy~(\ref{eq:grad_hk}).
It is also worth noting that these domains
constitute a subset of the class of Lipschitz domains used in Section~\ref{s:def_lipdom}. Thus, results
that hold for Lipschitz domains are also valid for this class of domains.

\newcommand{\Om}{\ensuremath{\Omega}}

\subsection{Lipschitz Domains}
\label{s:def_lipdom}
Solutions to elliptic partial differential equations in Lipschitz domains often belong to Hardy-type spaces.
Let~$\Om$ be a Lipschitz domain.
The truncated cones~$\Gamma(x')$ at~$x' \in \partial \Om$ are given by, e.g.,
\[
\Gamma(x') = \{ x \in \Om \sa  |x - x'| < 2 \mbox{dist}(x, \partial \Om) \}
\]
and the non-tangential maximal function is defined on the boundary~$\partial \Om$ by
\[
\N(u)(x') = \max_{k=1,2,\ldots,N} \sup \{ |u(x)| \sa x \in \Gamma(x') \cap B_k \} .
\]
The non-tangential convergence of~$u(x)$ to some number~$u(x')$ is defined as
\[
\lim_{\Gamma(x') \ni x \rightarrow x'} u(x) = u(x'), \quad x' \in \partial \Om,
\]
provided that the limit exists. Thus only approaches inside the cone~$\Gamma(x')$ are considered.
Let~$n(x')$ denote the normal vector at~$x'$ and
furthermore, if~$T$ is any tangential vector of~$\Om$ at~$x'$, the tangential gradient~$\nabla_T u$
with respect to~$T$ is defined as~$\nabla u \cdot T$. We refer to Kenig~\cite{Kenig1994} for further details.
The next two lemmas consists of known results which we prove for completeness sake.

\begin{lemma}
\label{l:lap_u_hom}
If~$g \in L^2(\partial \Om)$, where~$\Om \subset \Dd$ is a Lipschitz domain, then there exists a unique 
function~$u \in H^1(\Om)$
such that~$(1 - \Delta)u = 0$ in~$\Om$ and~$\partial_{\nu} u = g$ on~$\partial \Om$ in the 
sense that~$n \cdot \nabla u \rightarrow g$
nontangentially at almost every point on~$\partial \Om$, where~$n$ is the outwards normal. Moreover,
\[
\| \N(u) \|_{L^2(\partial \Om)} + \| \N(\nabla u) \|_{L^2(\partial \Om)} \leq C \| g \|_{L^2(\partial \Om)},
\]
where the constant~$C$ depends only on~$M$ and~$B_1,B_2,\ldots,B_N$
and the tangential gradient~$\nabla_T u$
exists in~$L^2(\partial \Om)$ in the sense of a weak limit in~$L^2$ of mean value 
integrals~$(\nabla_T u)_r$ 
{\rm(}see Section~1.8 of Kenig~\cite{Kenig1994}\/{\rm)}. 
\end{lemma}

\begin{proof}
The problem~$(1 - \Delta) u = 0$ in~$\Om$ and~$\partial_{\nu} w = g$ on~$\partial \Om$ 
has a weak solution~$w \in H^1(\Om)$ for every~$g \in L^2(\partial \Om)$ such that
\[
\| u \|_{H^1(\Om)} \leq C \| g \|_{L^2(\partial \Om)}, 
\]
where~$C$ is independent of~$g$ and~$u$. Let us extend~$u$ to a function~$\widetilde{u} \in H^1(\BR)$
with compact support 
such that~$\| \widetilde{u} \|_{H^1(\BR)} \leq C \| u \|_{H^1(\Om)}$. Put~$u = u_0 + u_1$,
where~$\Delta u_0 = \widetilde{u}$ on~$\BR$ and~$u_0 = 0$ on~$\partial \BR$. 
Then~$u_0 \in H^2(\BR)$ and
\[
\| u_0 \|_{H^2(\BR)} \leq C \|g\|_{L^2(\partial \Om)}.
\]
We also obtain that~$\Delta u_1 = 0$ in~$\Om$ and~$\partial_{\nu} u_1 = h$
with~$h = \partial_{\nu} u - \partial_{\nu} u_0$ 
satisfying~$\| h \|_{L^2(\partial \Om)} \leq C \| g \|_{L^2(\partial \Om)}$.

Suppose that~$U = 1$. 
Then~$\Delta U = 0$ and~$U = 1$ on~$\partial \Om$, and by Green's formula,
\[
\begin{aligned}
\int_{\partial \Om} \bigl( \partial_{\nu} u - \partial_{\nu} u_0 \bigr) U dS
& = 
\int_{\Om} \bigl( \nabla( u - u_0 ) \cdot \nabla U + \Delta(u - u_0) U \bigr) dx\\
& = 
\int_{\Om} \bigl( u - \widetilde{u} \bigr) dx = 0.
\end{aligned}
\]
The homogeneous Neumann problem~$\Delta u_1 = 0$ in~$\Om$ with~$\partial_{\nu} u_1 = h$ on~$\partial \Om$ 
has a unique solution~$u_1 \in H^1(\Om)$ such that
\begin{equation}
\label{eq:l:est_u1}
\| \N(u_1) \|_{L^2(\partial \Om)} + \| \N(\nabla u_1) \|_{L^2(\partial \Om)} 
	\leq C \| g \|_{L^2(\partial \Om)},
\end{equation}
where~$\N$ is the non-tangential maximal function; see Jerison and Kenig~\cite{Jerison1981}. 
Now,~(\ref{eq:l:est_u1}) and the fact that~$u = u_0 + u_1$ imply that 
\begin{equation*}
\| \N(u) \|_{L^2(\partial \Om)} + \| \N(\nabla u) \|_{L^2(\partial \Om)} 
	\leq C \| g \|_{L^2(\partial \Om)}. 
\end{equation*}
For the convergence of the tangential gradient, see Kenig~\cite{Kenig1994}.
\end{proof}

\begin{lemma}
\label{l:defK}
If~$f \in L^2(\Om)$, where~$\Om \subset \Dd$ is a Lipschitz domain, then there exists a unique function~$u \in H^1(\Om)$
such that~$(1 - \Delta)u = f$ in~$\Om$, and~$\partial_{\nu} u = 0$ on~$\partial \Om$ in the nontangential sense. Moreover,
\begin{equation}
\label{eq:est_max_sol}
\| \N(u) \|_{L^2(\partial \Om)} + \| \N(\nabla u) \|_{L^2(\partial \Om)} \leq C \| f \|_{L^2(\Om)},
\end{equation}
where the constant~$C$ depends only on~$M$ and~$B_1,B_2,\ldots,B_N$.
\end{lemma}

\begin{proof}
Extend~$f \in L^2(\Om)$ by zero to a function~$\widetilde{f} \in L^2(\Dd)$. 
Let~$v \in H^2(\Dd$) be the solution to~$(1 - \Delta) v = \widetilde{f}$ and~$v = 0$ on~$\partial \Dd$ 
such that
\begin{equation}
\label{eq:l:vf}
\| v \|_{H^2(\Dd)} \leq C \| f \|_{L^2(\Om)},
\end{equation}
and put~$u = v + w$. 
It follows that~$(1 - \Delta) w = 0$ in~$\Om$ and~$\partial_{\nu} w = - \partial_{\nu} v$ on~$\partial \Om$. 
Since~$\nabla v \in H^1(\R^n)$ and~(\ref{eq:l:vf}) holds,
the trace~$\partial_{\nu} v \in L^2(\partial \Om)$ satisfies
\begin{equation}
\label{eq:c:pnv_est}
\| \partial_{\nu} v \|_{L^2(\partial \Om)} \leq C \| v \|_{H^1(\R^n)} \leq C \|f \|_{L^2(\R^n)}.
\end{equation}
Applying Lemma~\ref{l:lap_u_hom} with~$g = -\partial_{\nu} v$, we obtain the unique~$w \in H^1(\Om)$
such that~$(1-\Delta) w = 0$,~$\partial_{\nu} w = g$, and
\[
\| \N(w) \|_{L^2(\partial \Om)} + \| \N(\nabla w) \|_{L^2(\partial \Om)} \leq C \| f \|_{L^2(\Om)},
\]
where we used~(\ref{eq:c:pnv_est}). Since~$u = v+w$, we have now proved the statements in the lemma.
\end{proof}

\noindent
Notice that Lemmas~\ref{l:lap_u_hom} and~\ref{l:defK} imply that
\begin{equation}
\label{eq:est_max_Kj}
\| \N(\Kj u) \|_{L^2(\partial \Oj)} + \| \N(\nabla \Kj u) \|_{L^2(\partial \Oj)} \leq C \| u \|_{L^2(\Oj)}, \quad j=1,2.
\end{equation}

\subsection{Extension Operators} 
It will be necessary for our purposes to extend functions from either Lipschitz-
or $C^{1,\alpha}$-domains. The following result provides the possibility to accomplish this.

\begin{lemma}
\label{l:extension}
\mbox{}
\begin{enumerate}[
label={\rm(\roman{*})},
ref=\ref{l:extension}(\roman{*}), 
]
\item
\label{l:lip_ext}
Suppose that~$f \in H^1(\partial \Om)$ and~$g \in L^2(\partial \Om)$,
where~$\Om$ is a Lipschitz domain. Then 
there exists a function~$u \in H^1(\Om^c)$ such that~$u \rightarrow f$
and~$n \cdot \nabla u \rightarrow g$ nontangentially at almost every point on~$\partial \Om$,
where~$n$ is the outwards normal of~$\Om$,
and there exists a constant~$C$ such that
\[
\| \N(u) \|_{L^2(\partial \Om)} + \|\N(\nabla u)\|_{L^2(\partial \Om)}
\leq
C ( \| f \|_{H^1(\partial \Om)} + \|g\|_{L^2(\partial \Om)} ),
\]
where~$C$ depends on~$M$ and~$B_1,B_2,\ldots,B_N$.
\item 
\label{l:holder_ext}
Suppose that~$f \in C^{1,\alpha}(\partial \Om)$ and~$g \in C^{0,\alpha}(\partial \Om)$,
where~$\Om$ is a~$C^{1,\alpha}$ domain. 
Then there exists a function~$u \in C^{1,\alpha}(\Om^c)$ such that~$u = f$ and~$\partial_{\nu} u = g$ on~$\partial \Om$,
and there exists a constant~$C$ such that
\begin{equation}
\label{eq:holder_ext}
\| u \|_{C^{1,\alpha}(\Om^c)} \leq C ( \| f \|_{C^{1,\alpha}(\partial \Om)} + \| g \|_{C^{0,\alpha}(\partial \Om)}). 
\end{equation}
\end{enumerate}
\end{lemma}

\begin{proof}
Let~$B_k$ be given as in Section~\ref{s:def_lipdom}. 
Choose~$\eta_k \in C^{\infty}_c(B_k)$,~$k=1,2,\ldots,N$, such that
$\eta_1 + \eta_2 +\cdots+\eta_N = 1$
in an open neighborhood containing~$\partial \Om$.
For each~$k$, define~$f_k = \eta_k f$ and~$g_k = \eta_k g$ on~$B_k \cap \partial \Om$,
and let~$f_k = g_k = 0$ on~$\partial B \cap \Om^c$.
Let~$D_k$ be the bounded domain with boundary~$(\partial \Om \cap B) \cup (\partial B \cap \Om^c)$.
Then~$D_k$ is a Lipschitz domain with connected boundary,~$f_k \in H^1(\partial D_k)$,
and~$g_k \in L^2(\partial D_k)$.  
According to, e.g., Dahlberg et al.~\cite{biharmonic}, there exists a solution~$u$
to~$\Delta^2 u = 0$ in~$D_k$ such 
that~$u \rightarrow f_k$ and~$n \cdot \nabla u \rightarrow g_k$ 
nontangentially at almost every point on~$\partial D_k$, where~$-n$ is the outwards normal at~$\partial D_k$.
Moreover,
\begin{equation}
\label{eq:biharm_est}
\begin{aligned}
\| \N(u) \|_{L^2(\partial D_k)} + \| \N(\nabla u) \|_{L^2(\partial D_k)} 
	&\leq C ( \| f_k \|_{H^1(\partial D_k)} + \|g_k\|_{L^2(\partial D_k)} )\\
	&\leq C ( \| f \|_{H^1(\partial \Om)} + \|g\|_{L^2(\partial \Om)} ),
\end{aligned}
\end{equation}
where~$C$ is independent of~$u$,~$f$, and~$g$, but depends on the Lipschitz constant of~$D_k$.
Carrying out the same argument for all of the balls~$B_k$ in Section~\ref{s:def_lipdom}, which is a finite number,
we obtain~$u \in H^1(D)$, where~$D = D_1 \cup D_2 \cup \cdots \cup D_m$.
We may extend~$u$ to all of~$\Om^c$ be letting~$u=0$ outside~$D$ and obtain~$u \in H^1(\Om^c)$
which satisfies the statement in~\ref{l:lip_ext}. 

The proof of Lemma~\ref{l:holder_ext} can be carried out analogously with the Lipschitz case. However, the result
is well known for~$C^{1,\alpha}$-domains and the proof is omitted. 
\end{proof}

We will commonly denote the extension for, e.g., a function~$u$, obtained from this Lemma by~$\widetilde{u}$.


\subsection{Determination of the Quantity~\boldmath{$\eps$}}
\label{s:lipeps}
We now proceed by determining a quantity~$\eps$ suitable for our purpose.
Let us investigate the assertions in~(\ref{eq:def_eps}) and~(\ref{eq:req_Xm}).
The assumption in~(\ref{eq:req_Xm}) is in our case  
\begin{equation}
\label{eq:applip_req_Xm}
\int_{\Oe \setminus \Ot} |\vp|^2 \, dx \leq \eps \| \vp \|^2_1
\quad
\mbox{for every } \vp \in \Xm.
\end{equation}
There exists a constant~$C$, depending on the domain~$\Oe$ and~$\lambda$, such that
for every weak solution to the elliptic problem~$(1-\Delta) \vp = \lambda \vp$ in~$\Oe$ 
with~$\partial_{\nu} \vp = 0$ on~$\partial \Oe$, 
\[
\| \vp \|_{L^{\infty}(\Oe)} \leq C \| \vp \|_{L^2(\Oe)};
\]
see, e.g, Theorem~8.15 in Gilbarg and Trudinger~\cite{Gilbarg2001}. 
This enables us to estimate the left-hand side of~(\ref{eq:applip_req_Xm}) by
\begin{equation}
\label{eq:est_vp2_subdom}
\begin{aligned}
\int_{\Oe \setminus \Ot} |\vp|^2 \, dx \leq \| \vp \|^2_{L^{\infty}(\Oe \setminus \Ot)} |\Oe \setminus \Ot| 
\leq C d \,  \| \vp \|_1^2,
\end{aligned}
\end{equation}
where~$d$ is the Hausdorff distance between~$\Oe$ and~$\Ot$ and~$|\Oe \setminus \Ot|$ is
the Lebesgue measure of~$\Oe \setminus \Ot$.

To prove the assertion in~(\ref{eq:def_eps}), we use the following lemmas.

\begin{lemma}
\label{l:diff_cap}
Suppose that~$v = \Kt \St w - \St \Ke \Se w$, where~$w \in L^2(\Dd)$. 
Then~$v$ satisfies~$(1 - \Delta) v = 0$ in~$\Oe \cap \Ot$ and~$v \in H^1(\partial(\Oe \cap \Ot))$.
Moreover, there exists a positive constant~$C$, depending only on~$M$ and~$B_1,B_2,\ldots,B_N$, such that
\begin{enumerate}
\item[{\rm(i)}] if~$w \in L^2(\Oe \cap \Ot)$, then
\[
\| v \|_{H^1(\Oe \cap \Ot)}^2 \leq 
C d \, \|g \|_{L^2(\partial(\Oe \cap \Ot))} \| w \|_{L^2(\Oe \cap \Ot)},
\]
\item[{\rm(ii)}] if~$w \in L^2(\Oe \setminus \Ot)$, then
\[
\| v \|_{H^1(\Oe \cap \Ot)}^2 \leq 
C d^{1/2} \, \|g \|_{L^2(\partial(\Oe \cap \Ot))} \| w \|_{L^2(\Oe \setminus \Ot)},
\]
\item[{\rm(iii)}] and if~$w \in L^2(\Ot \setminus \Oe)$, then
\[
\| v \|_{H^1(\Oe \cap \Ot)}^2 \leq 
C d^{1/2} \, \|g \|_{L^2(\partial(\Oe \cap \Ot))} \| w \|_{L^2(\Ot \setminus \Oe)},
\]
\end{enumerate}
where~$w$ is extended by zero outside the respective domains, and~$g = \partial_{\nu} v$ on~$\partial (\Oe \cap \Ot)$.
\end{lemma}

\begin{proof}
Since~$v \in H^1(\Oe \cap \Ot)$ 
satisfies~$(1 - \Delta) v = 0$ in~$\Oe \cap \Ot$
and~$\partial_{\nu} v$ belongs to~$L^2(\partial(\Oe \cap \Ot))$,
Lemma~\ref{l:lap_u_hom} implies that
\begin{equation}
\label{eq:est_v}
\| N(v) \|_{L^2(\partial (\Oe \cap \Ot))} + \| N(\nabla v) \|_{L^2(\partial (\Oe \cap \Ot))}
\leq
C \| g \|_{L^2(\partial(\Oe \cap \Ot))}
\end{equation}
and that~$v \in H^1(\partial (\Oe \cap \Ot))$. 
Moreover, Lemma~\ref{l:extension}(i) ensures the existence of an extension~$\vve \in H^1(\R^n)$ such that
\begin{equation}
\label{eq:est_ve}
\| N(\vve) \|_{L^2(\partial (\Oe \cap \Ot))} + \| N(\nabla \vve) \|_{L^2(\partial (\Oe \cap \Ot))}
\leq
C \| g \|_{L^2(\partial(\Oe \cap \Ot))}.
\end{equation}
Now,
\begin{equation*}
\begin{aligned}
\int_{\Oe \cap \Ot} \bigl( v^2 + |\nabla v|^2 \bigr) \, dx 
= {} & 
\int_{\partial (\Oe \cap \Ot)} v \partial_{\nu} v \, dS\\
= {} &
\int_{\partial \Oe \cap \Ot}  v \partial_{\nu} \Kt \St w \, dS
-
\int_{\Oe \cap \partial \Ot}  v \partial_{\nu} \Ke \Se w \, dS\\
= {} &
-\int_{\partial (\Ot \setminus \Oe)}  \vve \partial_{\nu} \Kt \St w \, dS
+\int_{\partial (\Oe \setminus \Ot)}  \vve \partial_{\nu} \Ke \Se w \, dS,
\end{aligned}
\end{equation*}
where we used the fact that~$\partial_{\nu} \Kt \St w = 0$ on~$\partial \Ot$ 
and~$\partial_{\nu} \Ke \Se w = 0$ on~$\partial \Oe$.

Since~$(1 - \Delta) \Kt \St w = \St w$ in~$\Ot \setminus \Oe$, we obtain that
\begin{equation}
\label{eq:lanki1}
-\int_{\partial (\Ot \setminus \Oe)}  \vve \partial_{\nu} \Kt \St w \, dS
=
\int_{\Ot \setminus \Oe}  \bigl( \vve \St w - \vve \Kt \St w - \nabla \vve \cdot \nabla \Kt \St w \bigr) \, dx.
\end{equation}
If~$w \in L^2(\Oe \cap \Ot)$, then~$\St w = 0$ and the right-hand side of~(\ref{eq:lanki1})
is bounded by
\begin{equation}
\label{eq:lanki2}
C d \, \| g \|_{L^2(\partial(\Oe \cap \Ot))} 
\| w \|_{L^2(\Oe \cap \Ot)} .
\end{equation}
This follows from the Cauchy-Schwarz inequality,~(\ref{eq:est_ve}), and~(\ref{eq:est_max_Kj}), since, e.g.,
\[
\begin{aligned}
\int_{\Ot \setminus \Oe}  |\vve \Kt \St w|  \, dx,
&\leq
\biggl( \int_{\Ot \setminus \Oe}  \vve^2  \, dx \biggr)^{1/2}
\biggl( \int_{\Ot \setminus \Oe}  (\Kt \St w)^2  \, dx \biggr)^{1/2}\\
&\leq C d \biggl( \int_{\partial (\Oe \cap \Ot)}  \N(\vve)^2  \, dx' \biggr)^{1/2}
\biggl( \int_{\partial (\Oe \cap \Ot)}  \N(\Kt \St w)^2  \, dx' \biggr)^{1/2}.
\end{aligned}
\]
If~$w \in L^2(\Oe \setminus \Ot)$, then~$\St w = 0$, and analogously with~(\ref{eq:lanki2}),
the expression in~(\ref{eq:lanki1}) is bounded by
$
C d \, \| g \|_{L^2(\partial(\Oe \cap \Ot))} 
\| w \|_{L^2(\Oe \setminus \Ot)} .
$
If~$w \in L^2(\Ot \setminus \Oe)$, then~$\St w = w$. Since
\[
\int_{\Ot \setminus \Oe}  | \vve w | \, dx
\leq C d^{1/2} \, \| g \|_{L^2(\partial(\Oe \cap \Ot))} 
\| w \|_{L^2(\Ot \setminus \Oe)},
\]
we obtain that~(\ref{eq:lanki1}) is bounded by
$
C d^{1/2} \, \| g \|_{L^2(\partial(\Oe \cap \Ot))} \| w \|_{L^2(\Ot \setminus \Oe)}.
$

Analogously, the expression
\[
\int_{\partial (\Oe \setminus \Ot)}  \vve \partial_{\nu} \Ke \Se w \, dS
=
\int_{\Oe \setminus \Ot}  \bigl( \vve \Ke \Se w + \nabla \vve \cdot \nabla \Ke \Se w - \vve \Se w \bigr) \, dx
\]
is bounded by
\[
\begin{aligned}
&C d \, 
\| g \|_{L^2(\partial(\Oe \cap \Ot))} 
\| w \|_{L^2(\Oe \cap \Ot)}
& & \mbox{if } w \in L^2(\Oe \cap \Ot),\\ 
&C d \| g \|_{L^2(\partial(\Oe \cap \Ot))} \| w \|_{L^2(\Ot \setminus \Oe)}
& & \mbox{if } w \in L^2(\Ot \setminus \Oe),\\
& C d^{1/2} 
\| g \|_{L^2(\partial(\Oe \cap \Ot))} 
\| w \|_{L^2(\Oe \setminus \Ot)}
& & \mbox{if } w \in L^2(\Oe \setminus \Ot),
\end{aligned}
\]
respectively.
\end{proof}

\begin{lemma}
\label{l:applip_maineq}
There exists a constant~$C > 0$ such that
\begin{equation}
\label{eq:applip_maineq}
\| \Kt w - \S \Ke \Ss w \|^2 \leq C d^{1/2} \, \| w \|^{2} \quad \mbox{for every } w \in L^2(\Ot)
\end{equation}
and
\begin{equation}
\label{eq:applip_Bu}
\| \B \vp \|^2 \leq C d \, \| \vp \|^{2} \quad \mbox{for every } \vp \in \Xm.
\end{equation}
\end{lemma}

\begin{proof}
Put~$v =  \Kt w - \S \Ke \Ss w $. We split the domain~$\Ot$ in two disjoint subdomains: 
$\Oe \cap \Ot$ and~$\Ot \setminus \Oe$.
For the subdomain~$\Ot \setminus \Oe$, it is clear from~(\ref{eq:est_max_Kj}) that
\begin{equation}
\label{eq:est_vt_diff}
\int_{\Ot \setminus \Oe} v^2 \, dx = \int_{\Ot \setminus \Oe} (\Kt w)^2 \, dx
\leq
C d \, \| w \|_{L^2(\Ot)}^2.
\end{equation}
Lemma~\ref{l:diff_cap} now implies the inequality in~(\ref{eq:applip_maineq})
since
\begin{equation}
\label{eq:lip_est_g}
\| g \|_{L^2(\partial(\Oe \cap \Ot)} \leq C \| w \|_{L^2(\Ot)}.
\end{equation}

To prove~(\ref{eq:applip_Bu}), observe first that~(\ref{eq:est_vp2_subdom}) holds.
Thus, by letting~$w = \vp$, we can apply Lemma~\ref{l:diff_cap} with~$v = \B \vp$ 
and obtain that
\[
\int_{\Oe \cap \Ot} (\B \vp)^2 \, dx \leq C d \, \| g \|_{L^2(\partial (\Oe \cap \Ot))} \| \vp \|_{L^2(\Oe)}.
\]
Since also~$\B \vp = v$ on~$\Ot \setminus \Oe$, inequalities~(\ref{eq:est_vt_diff}) and~(\ref{eq:lip_est_g})
are applicable, which concludes the proof of~(\ref{eq:applip_Bu}).
\end{proof}

\noindent
Thus, by~(\ref{eq:applip_maineq}) and~(\ref{eq:est_vp2_subdom}), it is clear that we can 
choose~$\eps = C d^{1/2}$.
Furthermore, if~$\Ot$ is a subdomain of~$\Oe$, we obtain a bound depending on~$d$ instead of~$d^{1/2}$ for a 
general function~$w \in L^2(\Ot)$; this is a consequence of that fact that the term~$\| w \|_{L^2(\Ot \setminus \Oe)}$
vanishes in Lemma~\ref{l:diff_cap} when~$\Ot \subset \Oe$. 

\begin{remark}
If~$\Ot \subset \Oe$, then
\begin{equation}
\label{eq:applip_maineq_subset}
\ds \| \Kt w - \S \Ke \Ss w \|^2 \leq C d \, \| w \|^{2} \quad \mbox{for every } w \in L^2(\Ot).
\end{equation}
\end{remark}

\subsection{Main Results for Lipschitz Domains}

We now derive an expression for the right-hand side of~(\ref{eq:T2_tk}) and prove
that in comparison, the remainder is small. We will then use Theorem~\ref{t:asymp} to
obtain a result for eigenvalues of~$\Kt$ near~$\lm^{-1}$. 

\begin{lemma}
\label{l:est_BKt}
If~$w \in L^2(\Ot)$, then
\begin{equation}
\label{eq:BpKtw}
\begin{aligned}
\lm \int_{\Ot} \B\vp \Kt w \, dx = {} &
\int_{\Oe \setminus \Ot} \bigl( 
	(1 - \lm)  \Kttext  \vp
	+ \nabla \Kttext \cdot \nabla \vp 
	\bigr) \, dx\\
& -
\int_{\Ot \setminus \Oe} \bigl( 
	(\Kttext - \Kt w) \widetilde{\vp} 
	+ \nabla \Kttext \cdot \nabla \widetilde{\vp} 
	\bigr) \, dx,
\end{aligned}
\end{equation}
where~$\Kttext \in H^1(\R^n)$ is an extension of~$\Kt^2 w \in H^1(\Ot)$.
\end{lemma}

\begin{proof}
We proceed similarly with the proof of Lemma~\ref{l:applip_maineq}.
Since~$(1 - \Delta) \B \vp = 0$ in~$\Oe \cap \Ot$, we obtain using Green's formula that 
\[
\begin{aligned}
\int_{\Oe \cap \Ot} \B \vp \Kt w \, dx
= {} &
\int_{\Oe \cap \Ot} \B \vp (1 - \Delta) \Kt^2 w \, dx\\
= {} &
\int_{\partial(\Oe \cap \Ot)}
\bigl( \Kt^2 w \partial_{\nu} \B \vp - \B \vp \partial_{\nu} \Kt^2 w \bigr) \, dS \\
= {} & \int_{\partial \Oe \cap \Ot}
\bigl( \Kt^2 w \partial_{\nu} \Kt \S \vs  - \B \vp \partial_{\nu} \Kt^2 w \bigr) \, dS\\
& - \int_{\Oe \cap \partial \Ot}
 \Kt^2 w \partial_{\nu} \Ke \vp \, dS.
\end{aligned}
\]
Furthermore,~$(1 - \Delta) \Kt w = w$ in~$\Ot \setminus \Oe$ and~$\partial_{\nu} \Kt \S \vp = 0$
on~$\partial \Ot$. Thus,
\[
\begin{aligned}
\int_{\partial \Oe \cap \Ot}
 \Kt^2 w \partial_{\nu} \Kt\S\vp \, dS
= {} & - \int_{\partial (\Ot \setminus \Oe)}
 \Kt^2 w \partial_{\nu} \Kt\S\vp \, dS\\
= {} & 
- \int_{\Ot \setminus \Oe} \bigl( 
	\Kt^2 w \Kt \S \vp
	+ \nabla \Kt^2 w \cdot \nabla \Kt \S \vp \bigr) \, dx\\
	& - \int_{\Ot \setminus \Oe}  \Kt w \Kt \S \vp \, dx,
\end{aligned}
\]
and analogously, 
\[
\begin{aligned}
- \int_{\partial \Oe \cap \Ot}
\B \vp \partial_{\nu} \Kt^2 w \, dS
= {} &
\int_{\partial( \Ot \setminus \Oe)}
\Bext \vp \partial_{\nu} \Kt^2 w  \, dS\\
= {} & 
\int_{\Ot \setminus \Oe} \bigl( 
	\Bext\vp \Kt^2 w 
	+ \nabla \Bext\vp \cdot \nabla \Kt^2 w \bigr) \, dx\\
	& - \int_{\Ot \setminus \Oe} \Bext \vp \Kt w  \, dx,
\end{aligned}
\]
where~$\Bext u = \Kt \S u - \S \widetilde{\Ke u}$ for~$u \in L^2(\Oe)$
satisfies
\[
\begin{aligned}
\| \Bext u - \B u \|_{L^2(\Ot)}^2 
&= \int_{\Ot \setminus \Oe} | \widetilde{\Ke u} |^2 \, dx \\
&\leq C d \int_{\partial \Oe} |\N( \Ke u )|^2 \, dx' 
\leq C d \, \|u \|_{L^2(\Oe)}^2
\end{aligned}
\]
by Lemma~\ref{l:extension}(i) and inequality~(\ref{eq:est_max_Kj}),

Similar to the treatment of the previous boundary integrals,
it follows from the facts that~$(1 - \Delta) \Ke \vp = \vp$ in~$\Oe \setminus \Ot$ and~$\partial_{\nu} \Ke \vp = 0$
on~$\partial \Oe$, that
\[
\begin{aligned}
 - \int_{\Oe \cap \partial \Ot} \Kt^2 w \partial_{\nu} \Ke \vp \, dS
&= 
 \int_{\partial(\Ot \setminus \Oe)} \Kttext \partial_{\nu} \Ke \vp \, dS\\
&= 
\int_{\Oe \setminus \Ot} \bigl( 
	\Kttext \Ke \vp 
	+ \nabla \Kttext \cdot \nabla \Ke\vp 
	- \Kttext \vp \bigr) \, dx.
\end{aligned}
\]
We have now proved that
\[
\begin{aligned}
\int_{\Ot} \B\vp \Kt w \, dx = {} &
 \int_{\Oe \cap \Ot} \B\vp \Kt w \, dx + \int_{\Ot \setminus \Oe} \B\vp \Kt w \, dx\\
= {} &
\lm^{-1} \int_{\Oe \setminus \Ot} \bigl( 
	(1 - \lm)  \Kttext  \vp
	+ \nabla \Kttext \cdot \nabla \vp 
	\bigr) \, dx\\
& -
\lm^{-1} \int_{\Ot \setminus \Oe} \bigl( 
	(\Kt^2 w - \Kt w) \widetilde{\vp} 
	+ \nabla \Kt^2 w \cdot \nabla \widetilde{\vp} 
\bigr) \, dx. 
\end{aligned}
\]
This is the equality in~(\ref{eq:BpKtw}).
\end{proof}

\begin{lemma}
\label{l:est_KtB}
There exists a constant~$C > 0$ such that
\begin{equation}
\label{eq:est_KtB}
\| \Kt \B \vp \|_{L^2(\Ot)}^2 \leq C d^{\, 3/2} \, \| \vp \|_{L^2(\Oe)}^2
\quad \mbox{for every } \vp \in \Xj[m].
\end{equation}
\end{lemma}

\begin{proof}
Since~$\Kt \B \vp$ is a solution to~$(1 - \Delta)\Kt \B \vp = \B \vp$ with~$\partial_{\nu} \Kt \B \vp = 0$
on~$\partial \Ot$, we obtain that
\[
\int_{\Ot} \bigl( (\Kt \B \vp)^2 + |\nabla \Kt \B \vp|^2 \bigr) \, dx = 
\int_{\Ot} \B \vp \Kt \B \vp \, dx.
\]
Let~$\Kttext$ be given 
as in Lemma~\ref{l:est_BKt} with~$w = \B \vp$. Then
\[
\| N(\Kttext) \|_{L^2(\partial \Ot)} + \| N(\nabla \Kttext) \|_{L^2(\partial \Ot)} \leq C \| \B \vp \|_{L^2(\Ot)}
\]
and Lemma~\ref{l:est_BKt} implies that
\[
\int_{\Ot} \B \vp \Kt \B \vp \, dx \leq C d \, \| \B \vp \|_{L^2(\Ot)} \| \vp \|_{L^2(\Oe)}.
\]
Since~$\| \B \vp \|_{L^2(\Ot)} \leq C d^{1/2} \, \| \vp \|_{L^2(\Oe)}$ according to Lemma~\ref{l:applip_maineq},
we obtain the inequality in~(\ref{eq:est_KtB}). 
\end{proof}

We now have all the tools available to prove our main result for Lipschitz domains,
i.e., expressing the difference between eigenvalues~$\lm[m]^{-1}$ and~$\um[k]^{-1}$ in known terms.

\begin{proposition}
\label{p:lipschitz}
Suppose that~$\Oe$ and~$\Ot$ are Lipschitz domains in the sense of Section~\ref{s:domains}.
Then
\begin{equation}
\label{eq:lip_main}
\lm[m]^{-1} - \um[k]^{-1} = \tk + O(d^{3/2})
\quad \mbox{for } k = 1,2,\ldots,\Jm. 
\end{equation}
Here,~$\tau = \tk$ is an eigenvalue of
\begin{equation}
\label{eq:lip_main_eig}
\begin{aligned}
\tau \ip{\vp}{\vs}{1} = {} &
\lm^{-1} \int_{\Oe \setminus \Ot} \bigl( 
	(1 - \lm) \widetilde{\Kt \S \vp}  \vs
	+ \nabla \widetilde{\Kt \S \vp} \cdot \nabla \vs 
	\bigr) \, dx\\
& - 
\lm^{-1} \int_{\Ot \setminus \Oe} \bigl( 
	(1 - \lm) (\Kt \S \vp) \widetilde{\vs} 
	+ \nabla \Kt \S \vp \cdot \nabla \widetilde{\vs}  
	\bigr) \, dx 
\end{aligned}
\end{equation}
for all $\vs \in \Xj[m]$, where~$\vp \in \Xj[m]$. 
Moreover,~$\tk[1], \tk[2], \ldots, \tk[\Jm]$ in~{\rm(}\ref{eq:lip_main}{\rm)}
run through all eigenvalues of~{\rm(}\ref{eq:lip_main_eig}{\rm)} counting their multiplicities.
\end{proposition}

\begin{proof}
We express~$\Kt^2 \S \vs$ in terms of the operator~$\B$:
\[
\Kt^2 \S \vs = \Kt (\B \vs + \S \Ke \vs) = \Kt \B \vs + \B \Ke \vs + \S \Ke^2 \vs. 
\]
If~$\vs \in \Xj[m]$, then
\[
\Kt^2 \S \vs = \Kt \B \vs + \lm^{-1} \Kt \S \vs.
\]
Put~$\Kttext = \widetilde{\Kt \B \vs} + \lm^{-1} \widetilde{\Kt \S \vs}$,
where~$\widetilde{\Kt \B \vs}$ and~$\widetilde{\Kt \S \vs}$
are the extensions of~$\Kt \B \vs$ and~$\Kt \S \vs$, respectively, given by Lemma~\ref{l:extension}(i).
Then~$\Kttext \in H^1(\R^n)$ and
\begin{equation}
\label{eq:est_Kttext}
\begin{aligned}
\| N(\Kttext) \|_{L^2(\partial \Ot)} + \|N(\nabla \Kttext)\|_{L^2(\partial \Ot)}
\leq {} &
C \bigl( \| \B \vs \|_{L^2(\Ot)}  + \| \S \vs \|_{L^2(\Ot)} \bigr)\\
\leq {} & C \| \vs \|_{L^2(\Oe)}.
\end{aligned}
\end{equation} 
Lemma~\ref{l:est_BKt} and~(\ref{eq:est_Kttext}) proves that
\begin{equation}
\label{eq:est_BvpKtSvs}
\begin{aligned}
\lm^2 \int_{\Ot} \B \vp \Kt \S \vs \, dx
= {} &
\int_{\Oe \setminus \Ot} \bigl( 
	(1 - \lm) \widetilde{\Kt \S \vs}  \vp
	+ \nabla \widetilde{\Kt \S \vs} \cdot \nabla \vp 
	\bigr) \, dx\\
& -
\int_{\Ot \setminus \Oe} \bigl( 
	(1 - \lm) (\Kt \S \vs) \widetilde{\vp}  
	+ \nabla \Kt \S \vs \cdot \nabla \widetilde{\vp}  
	\bigr)  \, dx\\
&+ O(d^{3/2}) \|\vp\|_{L^2(\Oe)} \|\vs\|_{L^2(\Oe)}.
\end{aligned}
\end{equation}
Observe also that~(\ref{eq:est_BvpKtSvs}) implies that
\begin{equation}
\label{eq:rough_BvpKtSvs}
|\ip{\B\vp}{\Kt\S\vs}{2}| \leq C d \| \vp \|_{L^2(\Oe)}  \| \vs \|_{L^2(\Oe)}.
\end{equation}

Lemmas~\ref{l:applip_maineq} and~\ref{l:est_KtB} imply that
\[
\rho = \lm^{-1} \sup_{\| \vp \| = 1} \bigl( \| \Kt \B \vp \|^2 + \eps \| \B \vp\|^2 \bigr)
= O(d^{\,3/2}).
\]
Thus, Theorem~\ref{t:asymp} proves that we obtain
\begin{equation}
\label{eq:lip_aaa}
\mu_k^{-1} - \lm[m]^{-1} = \tk + O(\rho + |\tk|d^{1/2}) = \tk + O(d^{\,3/2}) 
\end{equation}
since
\begin{equation}
\label{eq:lip_aaa2}
\tk \ip{\vp}{\vs}{2} = \tk \ip{\S \vp}{\S \vs}{2} + O(|\tk|d) = \lm \ip{\B\vp}{\Kt\S\vs}{2} + O(d^{\,2}).
\end{equation}
Now, equations~(\ref{eq:lip_aaa}),~(\ref{eq:lip_aaa2}), and~(\ref{eq:est_BvpKtSvs}), imply~(\ref{eq:lip_main}).
\end{proof}

\noindent From~(\ref{eq:rough_BvpKtSvs}) we obtain the following corollary.

\begin{corollary}
\label{c:lipschitz}
With the same assumptions as in Proposition~\ref{p:lipschitz}, there exists a constant~$C$, independent of~$d$, such that
\[
|\lm[m]^{-1} - \um[k]^{-1}| \leq C d 
\]
for~$k = 1,2,\ldots,\Jm[m]$.
\end{corollary}

If~$\Ot \subset \Oe$, 
the solution~$v_{\vp} = \B \vp$ to 
$(1 - \Delta) v_{\vp} = 0$ and~$\partial_{\nu} v_{\vp} = -\lm^{-1} \partial_{\nu} \vp$ for~$\vp \in \Xj[m]$ 
can be used to formulate the results above in terms of this solution. This can be an advantage since in
many cases these type of partial differential equations are well studied and explicit solutions or estimates
for solutions are known. Moreover, we also present an example
in Section~\ref{s:ex} based on this proposition, proving that the condition~$\alpha > 0$ is sharp for our result
in the~$C^{1,\alpha}$-case.

\begin{proposition}
\label{p:lipschitz_subdom}
Suppose that~$\Oe$ and~$\Ot$ are Lipschitz domains in the sense of Section~\ref{s:domains} and that~$\Ot \subset \Oe$.
Then
\begin{equation}
\label{eq:lip_main_subdom}
\lm[m]^{-1} - \um[k]^{-1} = \tk + O(d^{3/2})
\quad \mbox{for } k = 1,2,\ldots,\Jm. 
\end{equation}
Here,~$\tau = \tk$ is an eigenvalue of
\begin{equation}
\label{eq:lip_main_eig_subdom}
\begin{aligned}
\tau \ip{\vp}{\vs}{1} = {} &
\int_{\Ot} \bigl( \lm v_{\vp} v_{\vs} + v_{\vp} \vs \bigr) \, dx
\end{aligned}
\end{equation}
for all $\vs \in \Xj[m]$, where~$\vp \in \Xj[m]$. 
Moreover,~$\tk[1], \tk[2], \ldots, \tk[\Jm]$ in~{\rm(}\ref{eq:lip_main_subdom}{\rm)}
run through all eigenvalues of~{\rm(}\ref{eq:lip_main_eig_subdom}{\rm)} counting their multiplicities.
\end{proposition}


\subsection{The Case of a \boldmath{$C^{1,\alpha}$} Domain}
\label{s:holder}

We now consider the case when~$\Oe$ and~$\Ot$ are~$C^{1,\alpha}$ domains, where~$0 < \alpha < 1$.

\begin{lemma}
\label{l:holder_sol}
If~$\Oe$ is a $C^{1,\alpha}$-domain, then for every~$u \in L^{\infty}(\R^n)$,
$\Ke \Se u$ belongs to~$C^{1,\alpha}(\Oe)$.
\end{lemma}

\begin{proof}
This follows from the results in Section~9 of Agmon et al.~\cite{Agmon1959}.
\end{proof}

\begin{lemma}
\label{l:holder_g}
There exists a constant~$C > 0$ such that
\begin{equation}
\label{eq:holder_g}
\sup_{x' \in \partial (\Oe \cap \Ot)} |\partial_{\nu} \Kj \Sj w (x')| \leq C \| w \|_{L^2(\Oe \cap \Ot)} \, d^{\alpha}, \quad j=1,2,
\end{equation}
for every~$w \in L^2(\Oe \cap \Ot)$.
\end{lemma}

\begin{proof}
Let~$n_j$ be the outwards normal on~$\partial \Oj$ for $j=1,2$. 
On the boundary~$\partial \Oe$,~$\partial_{\nu} \Ke \Se w = 0$,
and on~$\partial \Ot$,~$\partial_{\nu} \Kt \St w = 0$.
We prove~(\ref{eq:holder_g}) for~$j=2$. The proof when~$j=1$ is analogous.
Thus,
\[
\partial_{\nu} \Kt \St w = n_1 \cdot \nabla \Kt \St w = (n_1 - n_2) \cdot \nabla \Kt \St w + n_2 \cdot \nabla \Kt \St w,
\]
and since it is clear that~$n_2 \cdot \nabla \Kt \St w = 0$ on~$\partial \Ot$,
\[
\sup_{\partial (\Oe \cap \Ot)} |\partial_{n_1} \Kt \St w| \leq C d^{\alpha} \, \| \nabla \Kt \St w \|_{L^2(\Ot)} 
\leq C d^{\alpha} \, \| w \|_{L^2(\Oe \cap \Ot)}.
\]
Here, we also used the fact that there exists a constant~$C$, independent of~$\Kt \St w$, such
that~$\| \nabla \Kt \St w \|_{L^{\infty}(\Ot)} \leq C \| \Kt \St w \|_{H^1(\Ot)}$.
Moreover, 
\[
|n_1 - n_2| \leq |\nabla (h_k^{(1)} - h_k^{(2)})| \leq C d^{\alpha},
\]
so we obtain~(\ref{eq:holder_g}) for~$j=2$.
\end{proof}

\noindent
We can use Lemma~\ref{l:holder_g} to refine the estimates provided in Lemma~\ref{l:applip_maineq}.

\begin{lemma}
\label{l:applip_maineq_holder}
There exists a constant~$C > 0$ such that
\begin{equation}
\label{eq:applip_maineq_holder}
\| \Kt w - \S \Ke \Ss w \|^2 \leq C d^{1/2+\alpha} \, \| w \|^{2} \quad \mbox{for every } w \in L^2(\Ot)
\end{equation}
and
\begin{equation}
\label{eq:applip_Bu_holder}
\| \B \vp \|^2 \leq C d^{1+\alpha} \, \| \vp \|^{2} \quad \mbox{for every } \vp \in \Xm.
\end{equation}
\end{lemma}

\begin{proof}
Proceeding as in Lemma~\ref{l:applip_maineq}, we obtain
the inequality in~(\ref{eq:applip_maineq_holder}) and also that
\[
\| \B \vp \|_{L^2(\Oe \cap \Ot)}^2 \leq C d^{1+\alpha} \|\vp\|_{L^2(\Oe)}^2
\]
since~(\ref{eq:holder_g}) implies that
\begin{equation}
\label{eq:lip_est_g_holder}
\| g \|_{L^2(\partial(\Oe \cap \Ot)} \leq C d^{\alpha} \| w \|_{L^2(\Ot)},
\end{equation}
where~$g$ is as in Lemma~\ref{l:diff_cap}.

In~$\Ot \setminus \Oe$,~$\B \vp = \Kt \S \vp$. Thus,~$\B \vp$ is a solution to~$(1 - \Delta) \B \vp = \S \vp = 0$
in~$\Ot \setminus \Oe$ such that~$\partial_{\nu} \B \vp = 0$ on~$\partial \Ot \cap \Oe^{c}$
and~$\partial_{\nu} \B \vp = \partial_{\nu} \Kt \S \vp$ on~$\Ot^{c} \cap \partial \Oe$. 
Lemma~\ref{l:holder_g} with~$w = \S \vp$ now implies that
\[
\| \partial_{\nu} \B \vp \|_{L^2(\partial(\Ot \setminus \Oe))} \leq C d^{\alpha} \, \|\vp \|_{L^2(\Oe)}
\]
and thus, Lemma~\ref{l:lap_u_hom} proves that
\[
\| \N(\B \vp) \|_{L^2(\partial(\Ot \setminus \Oe))} +
\| \N(\nabla \B \vp) \|_{L^2(\partial(\Ot \setminus \Oe))} 
\leq
C d^{\alpha} \, \| \vp \|_{L^2(\Oe)}.
\]
Hence,
\[
\int_{\Ot \setminus \Oe} (\B \vp)^2 \, dx \leq C d^{1+2\alpha} \| \vp \|_{L^2(\Oe)}^2. \qedhere
\]
\end{proof}

Since a $C^{1,\alpha}$ domain can be considered a Lipschitz domain, we know that the results 
from the previous section hold for~$\eps = C d^{1/2}$. 
However, in this case we may choose~$\eps = C  d^{\alpha + 1/2}$ if~$\alpha \leq 1/2$. 
This is clear from Lemma~\ref{l:applip_maineq_holder} and inequality~(\ref{eq:est_vp2_subdom}).
If~$\alpha > 1/2$, we may choose~$\eps = C d$. Inequality~(\ref{eq:est_vp2_subdom}) is 
the reason for the restriction on~$\alpha$. 

Similarly to the Lipschitz case, we shall employ Lemma~\ref{l:est_BKt} to obtain information
about the difference~$\mu_k^{-1} - \lm^{-1}$.
However, we wish to express the extension~$\Kttext$ in more explicit terms that depend directly on
the eigenfunction~$\vs$.

\begin{proposition}
\label{p:holder}
Suppose that~$\Oe$ is a $C^{1,\alpha}$ domain and~$\Ot$ is a perturbation in the sense of Section~\ref{s:domains} which satisfies~{\rm(}\ref{eq:grad_hk}{\rm)}. 
Then
\begin{equation}
\label{eq:holder_main}
\lm[m]^{-1} - \um[k]^{-1} = \tk + O(d^{1+\alpha})
\quad \mbox{for } k = 1,2,\ldots,\Jm. 
\end{equation}
Here,~$\tau = \tk$ is an eigenvalue of
\begin{equation}
\label{eq:holder_main_eig}
\begin{aligned}
\tau \ip{\vp}{\vs}{1} = {} &
\lm^{-2} \biggl(  
\int_{\Oe \setminus \Ot} \bigl( 
	(1 - \lm)  \vp \vs
	+ \nabla \vp \cdot \nabla \vs 
	\bigr) \, dx\\
& \qquad \qquad \quad -
\int_{\Ot \setminus \Oe} \bigl( 
	(1 - \lm) \widetilde{\vp} \widetilde{\vs}
	+ \nabla \widetilde{\vp} \cdot \nabla \widetilde{\vs} 
	\bigr) \, dx \biggr)
\end{aligned}
\end{equation}
for all $\vs \in \Xj[m]$, where~$\vp \in \Xj[m]$. 
Moreover,~$\tk[1], \tk[2], \ldots, \tk[\Jm]$ in~{\rm(}\ref{eq:holder_main}{\rm)}
run through all eigenvalues of~{\rm(}\ref{eq:holder_main_eig}{\rm)} counting their multiplicities.
\end{proposition}

\begin{proof}
We express~$\Kt^2 \S \vs$ in terms of the operator~$\B$:
\[
\Kt^2 \S \vs = \Kt (\B \vs + \S \Ke \vs) = \Kt \B \vs + \B \Ke \vs + \S \Ke^2 \vs. 
\]
If~$\vs \in \Xj[m]$, then
\[
\Kt^2 \S \vs = \Kt \B \vs + \lm^{-1} \Kt \S \vs.
\]
Put~$\Kttext = \widetilde{\Kt \B \vs} + \lm^{-1} \widetilde{\Kt \S \vs}$,
where~$\widetilde{\Kt \B \vs}$ and~$\widetilde{\Kt \S \vs}$
are the extensions of~$\Kt \B \vs$ and~$\Kt \S \vs$, respectively, given by Lemma~\ref{l:extension}(i).
Then~$\Kttext \in H^1(\R^n)$ and
\begin{equation}
\label{eq:est_Kttext_holder}
\begin{aligned}
\| N(\Kttext) \|_{L^2(\partial \Ot)} + \|N(\nabla \Kttext)\|_{L^2(\partial \Ot)}
\leq {} &
C \bigl( \| \B \vs \|_{L^2(\Ot)}  + \| \S \vs \|_{L^2(\Ot)} \bigr)\\
\leq {} & C \| \vs \|_{L^2(\Oe)}.
\end{aligned}
\end{equation} 
Moreover, $W = \lm^{-2} \widetilde{\vs} + \widetilde{\Kt \B \vs} + \lm^{-1} r$ on~$\Oe \setminus \Ot$ and~$\Ot \setminus \Oe$,
where~$r$ is defined as follows. 
On~$\Oe \cap \Ot$, we let~$r = \B \vs$. Then we extend~$r$ to~$\R^n$ such 
that~$r = \widetilde{\Kt \S \vs} - \lm^{-1} \vs$ in~$\Oe \setminus \Ot$
and~$r = \Kt \S \vs - \lm^{-1} \widetilde{\vs}$ in~$\Ot \setminus \Oe$,
where the extensions~$\widetilde{\Kt \S \vs}$ and~$\widetilde{\vs}$ are given
by Lemma~\ref{l:extension}{\rm(i)} and Lemma~\ref{l:extension}{\rm(ii)}, respectively.
It is now possible to use the same argument employed in the proof of Lemma~\ref{l:applip_maineq_holder} to obtain that
\[
\| N(r) \|_{L^2(\partial (\Oe \cap \Ot))} + \| N(\nabla r) \|_{L^2(\partial (\Oe \cap \Ot))}
\leq
C d^{\alpha} \, \| \vs \|_{L^2(\Oe)}.
\]
Now, this fact and the Cauchy-Schwarz inequality proves that
\begin{equation*}
\begin{aligned}
\int_{\Ot \setminus \Oe} \bigl( 
	|r  \widetilde{\vp}| + |\nabla r \cdot \nabla \widetilde{\vp}| \bigr) \, dx
\leq {} &
	C d^{1 + \alpha} \| \vs \|_{L^2(\Oe)} \| \vp \|_{L^2(\Oe)}.
\end{aligned}
\end{equation*}
Similarly, we can bound the corresponding integral over the domain~$\Oe \setminus \Ot$.
Using~(\ref{eq:applip_Bu_holder}), we can also refine the estimate given in~(\ref{l:est_KtB}):
\begin{equation}
\label{eq:est_KtB_holder}
\| \Kt \B \vs \|_{L^2(\Ot)}^2 \leq C d^{3/2 + \alpha/2} \, \|\vs \|_{L^2(\Oe)}^2
\leq C d^{1+\alpha} \, \| \vs \|_{L^2(\Oe)}^2.
\end{equation}
Thus, we obtain from Lemma~\ref{l:est_BKt} and inequality~(\ref{eq:est_Kttext_holder}) that
\begin{equation}
\label{eq:est_holder_BvpKtSvs}
\begin{aligned}
\lm \int_{\Ot} \B \vp \Kt \S \vs \, dx
= {} &
\lm^{-2} \int_{\Oe \setminus \Ot} \bigl( 
	(1 - \lm)  \vp \vs
	+ \nabla \vp \cdot \nabla \vs 
	\bigr) \, dx\\
& -
\lm^{-2} \int_{\Ot \setminus \Oe} \bigl( 
	(1 - \lm) \widetilde{\vp} \widetilde{\vs} +
	\nabla \widetilde{\vp} \cdot \nabla \widetilde{\vs} \bigr) \, dx\\
&+ O(d^{1+\alpha}) \|\vp\|_{L^2(\Oe)} \|\vs\|_{L^2(\Oe)}.
\end{aligned}
\end{equation}
Inequalities~(\ref{eq:est_KtB_holder}) and~(\ref{eq:applip_Bu_holder}) imply that
\begin{equation}
\label{eq:holder_aaa}
\rho =\lm^{-1} \sup_{\| \vp \| = 1} \bigl( \| \Kt \B \vp \|^2 + \eps \| \B \vp\|^2 \bigr)
= O(d^{\,(3+\alpha)/2})
\end{equation}
and thus, Theorem~\ref{t:asymp} proves that 
\[
\mu_k^{-1} - \lm[m]^{-1} = \tk + O(\rho + |\tk|\eps) = \tk + O(d^{\,1+\alpha}) 
\]
since we can choose~$\eps = C d^{\alpha + 1/2}$ if~$\alpha \leq 1/2$ and~$\eps = C d$ if~$\alpha > 1/2$, and
\begin{equation}
\label{eq:holder_aaa2}
\tk \ip{\vp}{\vs}{2} = \tk \ip{\S \vp}{\S \vs}{2} + O(|\tk|d) = \lm \ip{\B\vp}{\Kt\S\vs}{2} + O(d^{\,2}).
\end{equation}
Now, equations~(\ref{eq:holder_aaa}),~(\ref{eq:holder_aaa2}), and~(\ref{eq:est_holder_BvpKtSvs}), imply~(\ref{eq:holder_main}).
\end{proof}

Suppose that it is possible to characterize the perturbed domain~$\Ot$ by a function~$h$
defined on the boundary~$\partial \Oe$ such that~$(x',x_{\nu}) \in \partial \Ot$ is 
represented by~$x_{\nu} = h(x')$, where~$(x',0) \in \partial \Oe$ and~$x_{\nu}$ is the signed distance to the 
boundary~$\partial \Oe$ (with~$x_{\nu} < 0$ when~$x \in \Oe$).
The function~$h$ is assumed to be Lipschitz and satisfy~$|\nabla h| \leq C d^{\alpha}$. 
Thus, we obtain the following variation of Proposition~\ref{p:holder}.

\begin{corollary}
\label{c:holder_boundary}
Suppose that in addition to the assumptions of Proposition~\ref{p:holder}, the domain~$\Ot$ can be characterized by the function~$h$ as above.
Then
\begin{equation}
\label{eq:holder_main_boundary}
\begin{aligned}
\lm[m]^{-1} - \um[k]^{-1} = {} & 
\tk + O(d^{1+\alpha})
\end{aligned}
\end{equation}
for~$k = 1,2,\ldots,\Jm$. Here,~$\tau = \tk$ is an eigenvalue of
\begin{equation}
\label{eq:holder_main_boundary_eig}
\tau \ip{\vp}{\vs}{1} =
\lm^{-2} 
\int_{\partial \Oe} h(x') \bigl( 
	(1 - \lm)  \vp \vs
	+ \nabla \vp \cdot \nabla \vs 
	\bigr)  dS(x') \quad \mbox{for all } \vs \in \Xj[m],
\end{equation}
where~$\vp \in \Xj[m]$. Moreover,~$\tk[1], \tk[2], \ldots, \tk[\Jm]$ in~{\rm(}\ref{eq:holder_main_boundary}{\rm)}
run through all eigenvalues of~{\rm(}\ref{eq:holder_main_boundary_eig}{\rm)} counting their multiplicities.
\end{corollary}

\begin{proof}
We first prove that
\begin{equation}
\label{eq:est_holder_rdiff_vp}
\left\{
\begin{aligned}
&\sup_{(x',x_{\nu}) \in \Oe \setminus \Ot} |\vp(x',x_{\nu}) - \vp(x',0)|
\leq C d^{1+\alpha} \, \|\vp\|_{L^2(\Oe)},\\
&\sup_{(x',x_{\nu}) \in \Oe \setminus \Ot} |\nabla \vp(x',x_{\nu}) - \nabla \vp(x',0))| 
\leq C d^{\alpha} \, \|\vp\|_{L^2(\Oe)},
\end{aligned} \right.
\end{equation}
where the corresponding estimates hold for~$\widetilde{\vp}$ on~$\Ot \setminus \Oe$.
Since~$\vp \in C^{1,\alpha}(\Oe)$, it is clear that for~$x = (x',x_{\nu}) \in \Oe \setminus \Ot$,
\[
\vp(x',x_{\nu}) = \vp(x', 0) + x_{\nu} \partial_{\nu} \vp(x',0) + O(d^{1+\alpha}),
\]
where the remainder is bounded by~$C d^{1+\alpha} \, \|\vp\|_{L^2(\Oe)}$. This shows
that the first inequality in~(\ref{eq:est_holder_rdiff_vp}) is true. Similarly, 
the second inequality in~(\ref{eq:est_holder_rdiff_vp}) is also valid.
Thus,
\[
\int_{\Oe \setminus \Ot} \bigl( |\vp(x) - \vp(x',0)|^2 + 
	|\nabla \vp(x) - \nabla \vp(x',0)|^2 \bigr) \, dx
\leq
C d^{1 + \alpha} \, \| \vp \|_{L^2(\Oe)}^2,
\]
with the corresponding estimate for~$\widetilde{\vp}$ on~$\Ot \setminus \Oe$.
Hence, Proposition~\ref{p:holder} implies that
$\lm^{-1} - \um[k]^{-1}$ is given by
\[
\begin{aligned}
&\lm^{-2} \biggl(  
\int_{\partial \Oe \cap \Ot^c}  
	\int_{0}^{h(x')} \bigl( (1 - \lm)  \vp(x',0)^2
	+ |\nabla \vp(x',0)|^2 
	\bigr) \, dx_{\nu} \, dS(x')\\
& \qquad \qquad \quad -
\int_{\partial \Oe \cap \Ot} \int_0^{-h(x')} \bigl( 
	(1 - \lm) \widetilde{\vp}(x',0)^2
	+ |\nabla \widetilde{\vp}(x',0)|^2 
	\bigr)  \, dx_{\nu} \, dS(x') \biggr)\\
& \quad + O(d^{1+\alpha}).
\end{aligned}
\]
The desired conclusion follows from this statement.
\end{proof}


\subsection{Sharpness of the requirement~$\alpha > 0$ in Theorem~\ref{i:t:holder_boundary}}
\label{s:ex}
We now employ Proposition~\ref{p:lipschitz_subdom} to a specific Lipschitz perturbation
of a two dimensional cylinder. The aim here is also to show that Theorem~\ref{i:t:holder_boundary} is sharp
in the sense that~$\alpha > 0$ is necessary. 

Suppose that~$\eta:  \R \rightarrow \R$ is a periodic nonnegative Lipschitz continuous 
function such that~$\eta(t+1) = \eta(t)$ for all~$t \in \R$.
Let the rectangle~$\Oe$ in~$\R^2$ be defined by~$0 < x < T$ and~$0 < y < R$, where~$R$ and~$T$
are constants, and the subdomain~$\Ot \subset \Oe$ be defined by~$0 < x < T$ and~$\delt \eta(x / \delt) < y < R$,
where~$\delt = T / N$ for some large integer~$N$. We will consider boundary conditions periodic in~$x$
with Neumann data given on the curves~$y = \delt \eta(x/\delt)$ and~$y = R$. 

\begin{proposition}
\label{p:example}
For the domains~$\Oe$ and~$\Ot$ defined above,
\begin{equation}
\label{eq:ex_lip_main_subdom}
\lm[m]^{-1} - \um[k]^{-1} = \tk + O(\delt^2)
\quad \mbox{for } k = 1,2,\ldots,\Jm. 
\end{equation}
Here,~$\tau = \tk$ is an eigenvalue of
\begin{equation}
\label{eq:ex_lip_main_eig_subdom}
\begin{aligned}
\tau \ip{\vp}{\vs}{1} = {} &
\frac{\delt}{\lm^{2}} (\eta_0 + \eta_1) \int_{0}^{T} 
 \nabla \vp(x,0) \cdot \nabla \vs(x,0) \, dx\\
& + 
\frac{\delt(1 - \lm)}{\lm^{2}} \eta_0 \int_{0}^{T} 
\vp(x,0) \vs(x,0) \, dx
\end{aligned}
\end{equation}
for all $\vs \in \Xj[m]$, where~$\vp \in \Xj[m]$ and
\[
\eta_0 = \int_0^1 \eta(X) \, dX \qquad \mbox{and} \qquad \eta_1 = \int_0^1 V(X,\eta(X)) \eta'(X) \, dX.
\]
The function~$V$ is the solution to~$-\Delta_{X,Y} V = 0$ specified in~{\rm(}\ref{eq:V}{\rm)} below
and~$\eta_1$ is not zero if~$\eta$ is not identically constant.
Moreover,~$\tk[1], \tk[2], \ldots, \tk[\Jm]$ in~{\rm(}\ref{eq:ex_lip_main_subdom}{\rm)}
run through all eigenvalues of~{\rm(}\ref{eq:ex_lip_main_eig_subdom}{\rm)} counting their multiplicities.
\end{proposition}

To prove Proposition~\ref{p:example}, we will use Proposition~\ref{p:lipschitz_subdom}. To this end, we will
find the solution~$v$ to the problem
\begin{equation}
\label{eq:subdom_v}
(1 - \Delta)v = 0 \mbox{ in } \Ot, \qquad \partial_{\nu} v = -\lm^{-1} \partial_{\nu} \vp \mbox{ on } \gd \mbox{ and } \partial_{\nu} v = 0 \mbox{ on } \gR,
\end{equation}
and~$v$ is periodic in the first argument with period~$T$, that is,
\begin{equation}
\label{eq:subdom_v_period}
v(0,y) = v(T,y) \quad \mbox{and} \quad v'_x(0,y) = v'_x(T,y)
\qquad \mbox{for all } y \in (0,R).
\end{equation}
By~$\gd$ we denote the part of the boundary of~$\Ot$ where~$y = \delt \eta(x/\delt)$,
and by~$\gR$ the part where~$y = R$. Similarly,~$\gz$ is the part of~$\Oe$ where~$y = 0$.
The ansatz for the asymptotic expansion of~$v$ has the following form:
\begin{equation}
\label{eq:v_subdom_main}
v(x,y) = \delt w_0(x,y) + \delt V_0(X,Y;x) + \delt^2 V_1(X,Y;x) + \cdots,
\end{equation}
where~$w_0$,~$V_0$, and~$V_1$ are solutions to two model problems,
and the remainder consists of higher order terms. Since the construction of the asymptotic expansion of
the solution to problem~(\ref{eq:subdom_v}) is quite standard, we confine ourselves to only finding the
leading terms of this expansion. 
We have also introduced the new coordinates~$X = x/\delt$ and~$Y = y / \delt$.
Substituting~(\ref{eq:v_subdom_main}) into~(\ref{eq:subdom_v}), we obtain
\begin{equation}
\label{eq:subdom_v_subs}
\begin{aligned}
0 = {} & -\delt^{-1} \Delta_{X,Y} V_0(X,Y;x) \\
&- \Delta_{X,Y} V_1(X,Y;x) - 2\partial_{X}\partial_x V_0(X,Y;x)\\
&+ \delt \biggl( (1- \Delta_{x,y})w_0(x,y) + V_0(X,Y;x) \\
& \qquad \quad - \partial_x^2 V_0(X,Y;x) - 2\partial_{X}\partial_{x} V_1(X,Y;x) \biggr)
+ O(\delt^2)
\end{aligned}
\end{equation}
with the boundary condition
\begin{equation}
\label{eq:subdom_v_subs_bndry}
\begin{aligned}
-\lm^{-1} \nhat \cdot \nabla_{x,y} \vp(x,y) = {} & 
	\nhat \cdot \nabla_{X,Y} V_0(X,Y;x)\\
& + \delt \nhat \cdot \bigl( \nabla_{X,Y} V_1(X,Y;x) + \nabla_{x,y} w_0(x,y) \bigr)\\
& + \delt \nhat \cdot (\partial_x V_0(X,Y;x), \; 0) 
		 + O(\delt^2) 
\end{aligned}
\end{equation}
on~$\gd$, where
\begin{equation}
\label{eq:nhatz}
\nhat(X) = \frac{(\eta'(X), \; -1)}{\sqrt{1 + (\eta'(X))^2}}
\end{equation}
is the outwards normal on~$\gd$.

The function~$w_0$ is the solution with periodic boundary conditions in the sense of (\ref{eq:subdom_v_period})
to
\[
(1 - \Delta) w_0 = 0 \mbox{  in } \Oe
\quad \mbox{and} \quad \partial_{\nu} w_0 = g_{w_0} \in L^2(\gz) \mbox{ and } \partial_{\nu} w_0 = 0 \mbox{ on } \gR,
\]
where the Neumann data~$g_{w_0}$ will be specified below. 

Let~$\Ote$ be defined by~$0 < X < 1$ and~$\eta(X) < Y$.
We denote by~$\Gd$ the curve~$Y = \eta(X)$ for~$0 < X < 1$.
The functions~$V_j$,~$j=0,1$, will be solutions to the following model problem for right-hand sides specified below:
\begin{equation}
\label{eq:modelprob}
-\Delta_{X,Y} W = F \mbox{ in } \Ote \quad \mbox{and} \quad \partial_{\nu} W = G \mbox{ on } \Gd,
\end{equation}
and~$W$ is periodic in~$X$:
\[
W(0,Y) = W(1,Y) \quad \mbox{and} \quad W'_X(0,Y) = W'_X(1,Y) \mbox{ for all } Y.
\]
The functions~$G \in L^2(\Gd)$ and~$F$ satisfy 
\begin{equation}
\label{eq:orthogo}
\int_{\Ote} F \, dX dY + \int_{\Gd} G \, dS = 0
\end{equation}
and
\[
|F(X,Y)| \leq C e^{-b Y} \qquad \mbox{for some } b > 0.
\]
The orthogonality condition
above implies that this solution decays exponentially as~$Y \rightarrow \infty$:
\begin{equation}
\label{eq:est_V_j}
|W(X,Y)| \leq C e^{-a Y} \qquad \mbox{for some } a > 0.
\end{equation}

We now specify the boundary data for the model problems. 
Since~$\vp$ satisfies~$\vp'_y(x,0) = 0$ for all~$x$, a Taylor expansion yields
\[
\vp'_x(x,y) = \vp'_x(x,0) + O(y^2) = \vp'_{x}(x,0) + O(\delt^2)
\]
if~$y = \delt \eta(x/\delt)$. Similarly,
\[
\vp'_y(x,y) = \delt \eta(X) \vp''_{yy}(x,0) + O(\delt^2).
\]
Thus, 
\begin{equation}
\label{eq:ex_boundarycond}
\begin{aligned}
\lm^{-1} {\partial_{\nu} \vp}(x,y) = {} & \frac{\eta'(X) \vp'_x(x,0)}{\lm \sqrt{1 + (\eta'(X))^2}}
 - \delt \frac{\eta(X)\vp''_{yy}(x,0)}{\lm \sqrt{1 + (\eta'(X))^2}} + O(\delt^2)
\end{aligned}
\end{equation}
for~$(x,y) \in \gd$.

We now consider the variables~$X$,~$Y$, and~$x$ as independent. 
Equations~(\ref{eq:subdom_v_subs}) and~(\ref{eq:subdom_v_subs_bndry}) together with~(\ref{eq:ex_boundarycond}) imply that
\[
-\Delta_{X,Y} V_0 = 0 \mbox{ in } \Ote
\]
and
\begin{equation}
\label{eq:V_0_bndry}
\nhat \cdot \nXY V_0 = - \frac{\eta'(X)\vp'_x(x, 0)}{\lm \sqrt{1 + (\eta'(X))^2}} \mbox{ on } \Gd.
\end{equation}
It is clear that
\[
\int_{\Gd} \frac{\eta'(X)\vp'_x(x, 0)}{\lm \sqrt{1 + (\eta'(X))^2}}  \, dS(X) = \frac{\vp'_x(x,0)}{\lm} \int_0^1 \eta'(X) \, dX = 0.
\]
Thus equation~(\ref{eq:modelprob}) with~$F = 0$ and~$G$ equal to the right-hand side of~(\ref{eq:V_0_bndry})
has a solution~$V_0$ decaying exponentially as~$Y \rightarrow \infty$. The dependence on~$x$ can be described as follows.
Let~$V$ be the periodic (with respect to~$X$) solution to~(\ref{eq:modelprob}) with~$F = 0$ and
\begin{equation}
\label{eq:V}
G =  \frac{\eta'(X)}{1 + (\eta'(X))^2} \mbox{ on } \Gd.
\end{equation}
Then
\[
V_0(X,Y;x) = -\lm^{-1} V(X,Y) \vp'_{x}(x,0)
\] 
and
\[
\partial_{x} V_0(X,Y;x) = -\lm^{-1}V(X,Y) \vp''_{xx}(x,0).
\]

Similarly with above, equations~(\ref{eq:subdom_v_subs}),~(\ref{eq:subdom_v_subs_bndry}),
and~(\ref{eq:ex_boundarycond}) also imply that
\begin{equation}
\label{eq:V_1}
-\Delta_{X,Y} V_1 = 2 \partial_{X}\partial_{x} V_0(X,Y;x) \mbox{ in } \Ote
\end{equation}
and
\begin{equation}
\label{eq:V_1_bndry}
\nhat \cdot \nXY V_1 = \frac{\eta(X)\vp''_{yy}(x,0) - \eta'(X)\partial_x V_0(X,Y;x)}{\lm \sqrt{1 + (\eta'(X))^2}}
\mbox{ on } \Gd.
\end{equation}
Put~$F$ equal to the right-hand side of~(\ref{eq:V_1}) and~$G$ equal to the right-hand side of~(\ref{eq:V_1_bndry}).
We require the orthogonality condition in~(\ref{eq:orthogo}), so 
\[
0 = \int_{\Ote} 2\partial_{X} \partial_x V_0(X,Y;x) \, dX dY + \int_{\Gd} G \, dS.
\]
Furthermore, since~$\Delta V = 0$,
\begin{equation}
\label{eq:nVnV}
\begin{aligned}
\int_{\Ote} \partial_{X} V(X,Y) \, dX dY 
= {} &
\int_{\Gd} V(X,\eta(X)) \frac{\eta'(X)}{\sqrt{1 + (\eta'(X))^2}} \, dS \\
= {} &
\int_{\Ote} \Delta V(X,Y) V(X,Y) \, dX dY \\
&+ \int_{\Ote} \nabla V(X,Y) \cdot \nabla V(X,Y) \, dX dY \\
= {} &
\int_{\Ote} \nabla V(X,Y) \cdot \nabla V(X,Y) \, dX dY.
\end{aligned}
\end{equation}
Thus
\[
\begin{aligned}
0 &= \int_{\Ote} 2\partial_{X} \partial_x V_0(X,Y;x) \, dX dY + \int_{\Gd} G \, dS\\
&= 
\lm^{-1} \bigl( \eta_1 \vp''_{xx}(x,0) - \eta_0 \vp''_{yy}(x,0) \bigr)
- {w_0}'_{y}(x,0),
\end{aligned}
\]
where
\[
\eta_0 = \int_0^1 \eta(X) \, dX \qquad \mbox{and} \qquad \eta_1 = \int_0^1 V(X,\eta(X)) \eta'(X) \, dX.
\]
We note here that~$\eta_1 > 0$ if~$V$ is nonzero (due to~(\ref{eq:nVnV})), or equivalently, if~$\eta$ is not identically constant.
For the validity of~(\ref{eq:orthogo}), it is sufficient that
\begin{equation}
\label{eq:rand_w0}
{w_0}'_{y}(x,0) = \lm^{-1} \bigl( \eta_1 \vp''_{xx}(x, 0) - \eta_0 \vp''_{yy}(x,0) \bigr)
\end{equation}
and we have therefore obtained that the function~$g_{w_0}$ is given by the right-hand side of~(\ref{eq:rand_w0})
since~$\partial_{\nu} w_0 = -{w_0}'_{y}$ on the curve~$\gz$.

Now, from~(\ref{eq:lip_main_eig_subdom}) it is clear that we wish to simplify the expression
\[
\int_{\Ot} \bigl( \lm v_{\vp} v_{\vs} + v_{\vp} \vs \bigr) \, dx.
\]
With the current notation,~$v_{\vp} = v$ and~$v_{\vs}$ is the corresponding solution to~(\ref{eq:subdom_v}) 
with~$\vs$ instead of~$\vp$. Hence,
\[
\int_{\Ot}  \lm v_{\vp} v_{\vs}  \, dx = O(\delt^2).
\]
Thus, we consider
\[
\begin{aligned}
\int_{\Ot} v \vs \, dx &= \delt \int_{\Ot} w_0 \vs \, dx + \delt \int_{\Ot} V_0 \vs \, dx + \delt^2 \int_{\Ot} V_1 \vs \, dx
+  \cdots 
\end{aligned}
\]
From~(\ref{eq:est_V_j}) it follows that the integrals involving~$V_0$ and~$V_1$ are of
order~$O(\delt^2)$. For the first term,
\[
\int_{\Ot} w_0 \vs \, dx = \int_{\Oe} w_0 \vs \, dx - \int_{\Oe \setminus \Ot} w_0 \vs \, dx,
\]
where the second term is of order~$O(\delt)$. The first term can be expressed as
\[
\begin{aligned}
\int_{\Oe} w_0 \vs \, dx = {} & \int_{\partial \Oe} \partial_{\nu} w_0 \vs \, dS - \int_{\Oe} \nabla w_0 \cdot \nabla \vs \, dx\\
= {} & \int_{\partial \Oe} \partial_{\nu} w_0 \vs \, dS + \int_{\Oe} w_0 \Delta \vs \, dx,
\end{aligned}
\]
where we used the fact that~$\partial_{\nu} \vp = 0$ on~$\partial \Oe$. Moreover,~$\partial_{\nu} w_0 = 0$ at~$\gR$.
This implies that
\[
\begin{aligned}
\lm \int_{\Oe} w_0 \vs \, dx = {} & \int_{\gd} \partial_{\nu} w_0 \vs \, dS\\
= {} & -\lm^{-1} \int_{0}^T \bigl( \eta_1 \vp''_{xx}(x,0) - \eta_1 \vp''_{yy}(x,0) \bigr) dx\\
= {} & \lm^{-1} (\eta_0 + \eta_1) \int_{0}^T  \vp'_x(x,0)\vs'_x(x,0) \, dx\\
& + \lm^{-1}(1 - \lm) \eta_0 \int_{0}^T \vp(x,0)\vs(x,0)  \, dx,
\end{aligned}
\]
where we used the fact that~$\vp''_{xx} + \vp''_{yy} = (1-\lm)\vp$ and integration by parts.

Thus, we have obtained an expression for the right-hand side of~(\ref{eq:lip_main_eig_subdom}).
Moreover, the representation of~$v$ in~(\ref{eq:v_subdom_main}) implies that~$\| \Kt \B \vp\|^2 = O(\delt^2)$. 
This proves that Proposition~\ref{p:example} holds.

\def\bibname{References}

\end{document}